\documentclass{imanum}
\usepackage{graphicx}

\jno{drnxxx}
\usepackage{stmaryrd}
\SetSymbolFont{stmry}{bold}{U}{stmry}{m}{n}
\usepackage{amsmath}
\usepackage{amssymb}
\usepackage{amsthm}
\usepackage{tabularx}
\usepackage{subfigure}
\usepackage{epsfig}
\usepackage{indentfirst}
\usepackage{cases}
\usepackage{graphicx}
\usepackage{epsfig}
\usepackage{color}
\usepackage{float}
\usepackage{multirow}
\usepackage{algorithm}
\usepackage{algpseudocode}

\renewcommand{\vec}[1]{\mathbf{#1}}

\def\subFV{\scriptscriptstyle{FV}}
\def\subFS{\scriptscriptstyle{FS}}
\def\subVS{\scriptscriptstyle{VS}}

\newcommand{\be}{\begin{equation}}
\newcommand{\ee}{\end{equation}}
\newcommand{\ba}{\begin{array}}
\newcommand{\ea}{\end{array}}
\newcommand{\bea}{\begin{eqnarray}}
\newcommand{\eea}{\end{eqnarray}}
\newcommand{\beas}{\begin{eqnarray*}}
\newcommand{\eeas}{\end{eqnarray*}}

\newtheorem{thm}{Theorem}[section]
\newtheorem{prop}{Proposition}[section]

\begin{document}

\title{An energy-stable parametric finite element method for simulating solid-state dewetting}
\shorttitle{An energy-stable PFEM for solid-state dewetting }

\author{%
{\sc
Quan Zhao}\thanks{Email:quanzhao90@u.nus.edu}
\\[2pt]
Department of Mathematics, National University of
Singapore, Singapore, 119076\\[6pt]
{\sc Wei Jiang}\thanks{Corresponding author. Email: jiangwei1007@whu.edu.cn}\\[2pt]
School of Mathematics and Statistics,
Wuhan University, Wuhan, 430072, China\\[2pt]
Hubei Key Laboratory of Computational Science,
Wuhan University, Wuhan, 430072, China\\[6pt]
{\sc and}\\[6pt]
{\sc Weizhu Bao}\thanks{Email: matbaowz@nus.edu.sg; URL: http://blog.nus.edu.sg/matbwz/}
\\[2pt]
Department of Mathematics, National University of
Singapore, Singapore, 119076
}
\shortauthorlist{Q.~Zhao, W.~Jiang and W.~Bao}

\maketitle

\begin{abstract}
{We propose an energy-stable parametric finite element method (ES-PFEM) for
simulating solid-state dewetting of thin films in two dimensions via a sharp-interface model, which is governed by surface diffusion and contact line (point) migration together with proper boundary conditions. By reformulating the relaxed contact angle condition into a Robin-type boundary condition and then treating it as a natural boundary condition, we obtain a new variational formulation for the problem, in which the interface curve and its contact points are evolved simultaneously. Then, the variational problem is discretized in space by using piecewise linear elements. A full discretization is presented by adopting the backward Euler method in time, and the well-posedness and energy dissipation of the full discretization are
established. The numerical method is semi-implicit (i.e., a linear system to be solved at each time step and thus efficient), unconditionally energy-stable with respect to the time step, and second-order in space measured by a manifold distance between two curves. In addition, it demonstrates equal mesh distribution when the solution reaches its equilibrium, i.e., long-time dynamics. Numerical results are reported to show  accuracy and efficiency as well as some good properties of the proposed numerical method.
}
{Solid-state dewetting, surface diffusion, contact line migration, variational formulation, energy-stable parametric finite element method, manifold distance.}
\end{abstract}

\begin{center}
\textbf{Dedicated to the memory of John W. Barrett}
\end{center}

\section{Introduction}

Driven by capillarity effects, solid thin films deposited on substrate are often metastable or unstable in the as-deposited state, and could exhibit complicated morphological evolution when heated to a critical temperature even below the material's melting point. This phenomenon, called as solid-state dewetting~\citep[see][]{Thompson12}, has been widely observed in various thin film/substrate systems~\citep[see, e.g.,][]{Jiran90, Kim13, Rabkin14, Bollani19}. Recently, solid-state dewetting has demonstrated wide applications in thin film technologies. For example, it may be deleterious during fabricating thin film structures (e.g., microelectronic and optoelectronic devices), because it can destroy the devices' structure and reliability; on the other hand, it may be advantageous and can be positively used to create patterns of nanoscale particles, e.g., used in sensors~\citep[see][]{Armelao06} and as catalysts for carbon~\citep[see][]{Randolph07} and semiconductor nanowire growth~\citep[see][]{Schmidt09}. These important applications have inspired increasing interests in understanding the underlying mechanism of solid-state dewetting phenomena
\citep[see, e.g.,][]{Wong00,Dornel06,Jiang12,Wang15,Jiang16,Backofen18,Jiang19a,Jiang19b}.

\begin{figure}[htp!]
\centering
\includegraphics[width=0.9\textwidth]{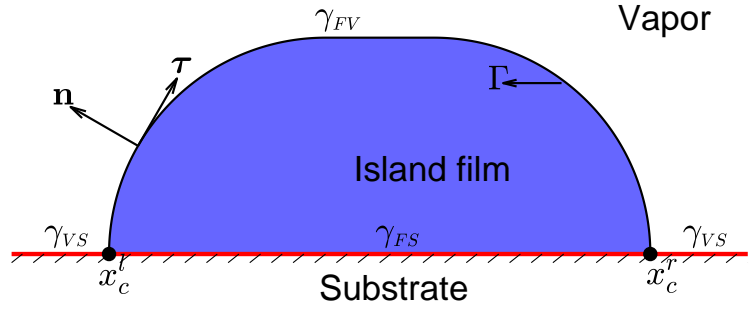}
\label{fig:model2d}
\caption{A schematic illustration of an island film on a rigid, flat substrate (i.e., the $x$-axis) in two dimensions, where $x_c^l$ and $x_c^r$ are the left
and right contact points, $\gamma_{_{\subFV}},~\gamma_{_{\subVS}}$ and $\gamma_{_{\subFS}}$ represent the film/vapor, vapor/substrate and film/substrate surface energy densities, respectively.}
\end{figure}

In general, the solid-state dewetting problem can be regarded as a type of open curves/surfaces evolution problems governed by surface diffusion and contact point/line migration (or moving contact lines) in two dimensions (2D)/three dimensions (3D)~\citep[see, e.g.,][]{Wong00,Jiang12,Wang15,Jiang19c}. As illustrated in Fig.~\ref{fig:model2d}, a contact point in 2D is a triple point where three phases (i.e., solid film, vapor and substrate) meet. As time evolves, the contact point moves along the substrate, and it brings an additional kinetic feature to this problem. Recently, different mathematical models and simulation methods have been proposed for simulating solid-state dewetting, such as sharp-interface models ~\citep[see][]{Wong00,Wang15,Jiang19c}, phase field models~\citep[see][]{Jiang12,Dziwnik15b,Naffouti17,Huang19b},
crystalline formulation method ~\citep[see][]{Carter95,Zucker13}, discrete surface chemical potential method~\citep[see][]{Dornel06} and
kinetic Monte Carlo method~\citep[see][]{Pierre09b}. In this paper, we mainly focus on how to design an efficient and accurate numerical method about solving a sharp-interface model for solid-state dewetting.

As shown in Fig.~\ref{fig:model2d}, the solid-state dewetting problem in 2D is described as the evolution of an open curve $\Gamma(t)=\vec X(s,t)=(x(s, ~t),\;y(s,~t))^T$, which is parameterized by arc length $s\in[0,~L(t)]$
with $L:=L(t)=|\Gamma(t)|$ the total length of the curve at time $t$,
and it intersects with the solid substrate (i.e., the $x$-axis) at the left and right moving contact points, i.e., $x_c^l(t):=x(0,t)$ and $x_c^r(t):=x(L,t)$. By using thermodynamic variation of the total interfacial energy, a dimensionless sharp-interface model for simulating solid-state dewetting of thin films with isotropic surface energy in 2D can be derived as~\citep[see][]
{Wang15, Bao17b}
\begin{subequations}
\label{eqn:2diso}
\begin{numcases}{}
\label{eqn:2diso1}
\partial_{t}\vec{X}=\partial_{ss}\kappa \; \vec{n}, \quad 0<s<L(t),\quad t>0, \\
\label{eqn:2diso2}
\kappa=-\left(\partial_{ss}\vec{X}\right)\cdot\vec{n}, \quad\vec n = (-\partial_sy,~\partial_sx)^T,
\end{numcases}
\end{subequations}
where $\vec n:=\vec n(s,t)$ is the unit outer normal vector of the curve,  and $\kappa:=\kappa(s,~t)$ represents the curvature of the curve.
The initial curve is given as
\begin{equation}
\vec{X}(s,0):=\vec{X}_0(s)=(x_0(s),~y_0(s))^T, \qquad 0\le s\le L_0:=L(0),
\label{eqn:initial}
\end{equation}
satisfying $x_0(0)<x_0(L_0)$.
The boundary conditions are given as
\begin{itemize}
\item [(i)] contact point condition
\begin{equation}
\label{eqn:weakBC1}
y(0,t)=0,\quad y(L,t) = 0,\quad t\geq0;
\end{equation}
\item [(ii)] relaxed contact angle condition
\begin{equation}
\label{eqn:weakBC2}
\frac{d x_c^l(t)}{d t}=\eta (\cos\theta_d^l - \sigma),\qquad
\frac{d x_c^r(t)}{d t}=-\eta (\cos\theta_d^r -\sigma),\qquad t\ge0;
\end{equation}
\item [(iii)]zero-mass flux condition
\begin{equation}
\label{eqn:weakBC3}
\partial_s\kappa(0,t) =0,\qquad\partial_s\kappa(L,t) = 0, \qquad t\geq 0;
\end{equation}
\end{itemize}
where the material constant $\sigma$ is defined as $\sigma: = \frac{\gamma_{_{\subVS}}-\gamma_{_{\subFS}}}{\gamma_{_{\subFV}}}$ with $\gamma_{_{\subFV}}$, $\gamma_{_{\subVS}}$ and $\gamma_{_{\subFS}}$ representing the film/vapor, vapor/substrate and film/substrate surface energy densities, respectively, which determines the equilibrium contact angle  $\theta_i\in(0,\pi)$ (i.e. the well-known isotropic Young's angle)  satisfying $\sigma=\cos\theta_i\in(-1,1)$; $\theta_d^l:=\theta_d^l(t)$ and $\theta_d^r:=\theta_d^r(t)$ are the (dynamic) contact angles at the left and right moving contact points, respectively;
and  $\eta>0$ is the contact line mobility which controls the relaxation rate of the dynamical contact angles $\theta_d^l$ and $\theta_d^r$ to the equilibrium contact angle $\theta_i$. In addition, we adopt $x_c^l(t)=x(s=0,t)$  and $x_c^r(t)=x(s=L(t),t)$ and assume that they satisfy
$x_c^l(t)\le x_c^r(t)$.

In fact, condition (i) implies that the two contact points must move along the flat substrate. When $\eta\to+\infty$, condition (ii) collapses to the well-known Young equation \citep[see][]{Young1805}. Condition (iii) is proposed to ensure that the total area/mass of the film is conserved during the evolution,
i.e., no-mass flux at the moving contact points. In addition, by defining
the total area/mass of the film $A(t)$ (i.e., the enclosed area by the curve $\Gamma(t)$ and the substrate) and the total interfacial free energy of the system $W(t)$  as
\begin{equation}\label{eqn:totalmass}
A(t)=\int_0^{L(t)}y(s,t)\partial_s x(s,t)\;ds,\qquad
W(t)=|\Gamma(t)| + W_{\rm sub} = L(t) - \sigma[x_c^r(t)-x_c^l(t)],\qquad t\ge 0,
\end{equation}
one can prove that \citep[see][]{Wang15, Bao17b}
\begin{equation}
\frac{d}{dt}A(t) = 0,\qquad \frac{d}{dt}W(t)= -\int_0^{L(t)}(\partial_s\kappa)^2\;ds - \frac{1}{\eta}\left[\left(\frac{dx_c^l}{dt}\right)^2+
\left(\frac{dx_c^r}{dt}\right)^2\right]\leq 0, \qquad t\ge0,
\label{eqn:massenergy2d}
\end{equation}
which immediately imply that the sharp-interface model (\ref{eqn:2diso}) with the boundary conditions (\ref{eqn:weakBC1})-(\ref{eqn:weakBC3}) and the initial condition \eqref{eqn:initial} satisfies mass conservation and energy dissipation, i.e.,
\begin{equation}
A(t) \equiv A(0),\qquad W(t)\le W(t^\prime)\le W(0), \qquad t\ge t^\prime\ge0.
\label{eqn:massenergy2dd}
\end{equation}

Different numerical methods have been proposed in the literature for
simulating the evolution of a closed or open curve under surface diffusion
(\ref{eqn:2diso}) including applications in solid-state dewetting. When the curve can be represented by a graph with specific boundary conditions, the finite element methods~\citep[see][]{Coleman96,Deckelnick03,Bansch04,Deckelnick05fully} have been proposed. However, these methods cannot be directly adopted to simulating the evolution of a closed curve or solid-state dewetting problems. On the other hand, an implicit finite difference method was proposed by \citet{Mayer01} for surface diffusion flow in 3D with adaptive triangular mesh. Meanwhile, the ``marker-particle'' method, i.e., an explicit finite difference scheme together with re-meshing at each time step, was proposed for solving the sharp-interface model of solid-state dewetting problems~\citep[see][]{Wong00,Du10,Wang15}. However, this method suffers from a very severe stability restriction and it is not easy to extend to 3D and/or
anisotropic surface energies.

Based on the previous work~\citep[see the recent review paper by][]{Barrett20},
and by reformulating the surface diffusion equation \eqref{eqn:2diso} into
\begin{subequations}
\label{eqn:2disonew}
\begin{numcases}{}
\label{eqn:2disonew1}
\vec{n}\cdot \partial_{t}\vec{X}-\partial_{ss}\kappa=0,  \\
\label{eqn:2disonew2}
\kappa\,\vec{n}+\partial_{ss}\vec{X}=0, \qquad 0<s<L(t),\quad t>0,
\end{numcases}
\end{subequations}
Barrett {\it et al.} introduced a novel variational formulation of \eqref{eqn:2disonew} and presented an elegant parametric finite element method (PFEM) for the evolution of a closed curve under surface diffusion~\citep[see][]{Barrett07, Barrett08JCP,Barrett19}. The PFEM has a few good properties including unconditional stability, energy dissipation and mesh equal distribution when the solution reaches its equilibrium. It has been successfully extended for simulating the evolution of curves with  grain boundary motions, thermal grooving and sintering, and triple junctions \citep[see][]{Barrett07, Barrett10}.
Recently, similar to the ``marker-particle'' method for dealing with the relaxed contact angle condition \eqref{eqn:weakBC2}~\citep[see][]{Wong00,Wang15}, by adopting the forward Euler method to discretize the relaxed contact angle condition \eqref{eqn:weakBC2} to first evolve the two contact points $x_c^l$ and $x_c^r$ and then treating the new positions of the two contact points as a Dirichlet-type boundary condition (or essential boundary condition) in deriving the variational formulation, we successfully extended the PFEM for simulating solid-state dewetting problems in 2D and 3D as well as the evolution of curves with triple junctions~\citep[see][]{Bao17,Jiang19a,Zhao19,Zhao19b}.
Many interesting phenomena of solid-state dewetting were observed
by using the proposed PFEM method~\citep[see][]{Bao17b,Jiang18a,Jiang19b,Zhao19b}. However, the method suffers from a few drawbacks: (i) it separately deals with the motions of the interface curve and its contact points, and does not make full use of the variational structure of the solid-state dewetting problem; (ii) the energy dissipation cannot be proved in the full discretization, and the stability condition depends on the mesh size and the contact line mobility; (iii) the convergence rate in space reduces to
first-order for the evolution of an open curve (i.e., solid-state dewetting problem), instead of second-order for the evolution of a closed curve~\citep[see Tables 1 and 3 in][]{Bao17}.

The main aim of this paper is to propose a new energy-stable parametric finite element method (ES-PFEM) for the solid-state dewetting problem
\eqref{eqn:2diso} with boundary conditions (\ref{eqn:weakBC1})-(\ref{eqn:weakBC3}) and initial condition \eqref{eqn:initial} by taking the following key ideas: (I) to reformulate the relaxed contact angle condition \eqref{eqn:weakBC2} into a Robin-type boundary condition; (II) to obtain a new variational problem
by adopting \eqref{eqn:2disonew} instead of \eqref{eqn:2diso} and treating
\eqref{eqn:weakBC2} as a natural boundary condition instead of an essential
boundary condition; and (III) to discretize the variational problem in space
by continuous piecewise linear elements and in time by (semi-implicit) backward Euler method.
We establish area/mass conservation and energy dissipation of the new
variational problem and its semi-discretization in space. We also prove
energy dissipation of the full discretization. The proposed ES-PFEM
overcomes all drawbacks of the previous PFEM for simulating solid-state dewetting.  Extensive numerical results are reported to demonstrate efficiency and accuracy as well as unconditional energy stability of the proposed ES-PFEM.

The rest of the paper is organized as follows. In section 2, we present a new variational formulation for the sharp-interface model and prove mass conservation and energy dissipation of the variational problem. In section 3, we propose a semi-discretization in space by using the continuous piecewise linear elements and establish its mass conservation, energy dissipation, mesh equidistribution property and long-time behavior. In section 4, we present a full discretization by adopting the backward Euler method in time and show its well-posedness, energy dissipation and long-time behavior. In section 5, extensive numerical results are reported to demonstrate the efficiency and accuracy of the proposed ES-PFEM. Finally, we draw some conclusions in section 6.

\section{A new variational formulation}

In this section, we present a new variational formulation for the sharp-interface model (\ref{eqn:2diso}) with boundary conditions (\ref{eqn:weakBC1})-(\ref{eqn:weakBC3}) and initial condition \eqref{eqn:initial} by adopting (\ref{eqn:2disonew}) and reformulating
\eqref{eqn:weakBC2} into a Robin-type boundary condition.
We also establish mass conservation and energy dissipation
of the new variational problem.

\subsection{The formulation}

As illustrated in Fig.~\ref{fig:contact}, we label the unit tangential vector of the curve as $\boldsymbol{\tau}:=\boldsymbol{\tau}(s,t)=(\partial_s x,~\partial_x y)^T$, and the unit vector along the substrate line as $\vec t_{\rm sub}=(1,~0)^T$.  Noticing
\begin{eqnarray}
\cos \theta_d^l(t) &=& \boldsymbol{\tau}(s,t)\Big|_{s=0}\cdot\vec t_{\rm sub}=\partial_s x(s,t)\Big|_{s=0},\\\cos \theta_d^r(t)& =& \boldsymbol{\tau}(s,t)\Big|_{s=L(t)}\cdot\vec t_{\rm sub}=\partial_s x(s,t)\Big|_{s=L(t)},
\end{eqnarray}
we can reformulate the relaxed contact angle condition \eqref{eqn:weakBC2}
into a Robin-type boundary condition as
\begin{equation}
\partial_s x(s,t) \Big|_{s=0} = \sigma + \frac{1}{\eta}\frac{dx_c^l(t)}{dt},\qquad \partial_s x(s,t) \Big|_{s=L(t)} = \sigma - \frac{1}{\eta}\frac{dx_c^r(t)}{dt}.
\label{eqn:Robin}
\end{equation}
In fact, when $\eta\rightarrow+\infty$, the Robin-type boundary condition \eqref{eqn:Robin} collapses to the conventional Neumann boundary condition (or the well-known isotropic Young equation) as
\begin{equation}
\label{eqn:young}
\partial_sx(s,t)\Big|_{s=0} = \partial_sx(s,t)\Big|_{s=L(t)} = \sigma.
\end{equation}

\begin{figure}[!htbp]
\centering
\includegraphics[width=0.8\textwidth]{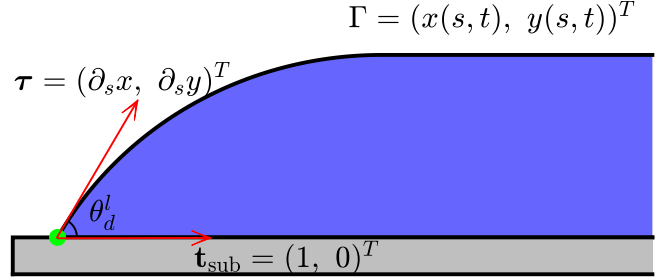}
\caption{A schematic illustration of the interface profile near the left contact point, where $\boldsymbol{\tau}:=\boldsymbol{\tau}(s,~t)=(\partial_s x,\partial_s y)^T$ is the unit tangential vector, and $\vec t_{\rm sub}=(1,~0)^T$ represents the unit vector along the $x$-coordinate (i.e., the substrate line).}
\label{fig:contact}
\end{figure}

In order to obtain a variational formulation of the sharp-interface model, for convenience, we introduce a time independent variable $\rho$ such that $\Gamma(t)$ can be parameterized over the fixed domain $\rho\in \mathbf{I}=[0,1]$ (here $\rho$ and $s$ can be respectively regarded as the Lagrangian and Eulerian variables of the curve $\Gamma(t)$, and
we do not distinguish $\vec X(\rho,t)$ and $\vec X(s,t)$ for representing $\Gamma(t)$ when there is no misunderstanding) as
\begin{equation}
\Gamma(t):=\vec X(\rho,t)=(x(\rho,t),~y(\rho,t))^T:\; \mathbf{I}\times [0, T]\;\rightarrow \;\mathbb{R}^2.
\end{equation}
Based on this parameterization, the arc length $s$ can be given as $s(\rho,t)=\int_0^\rho |\partial_q\vec{X}|\;dq$, and we have $\partial_\rho s=|\partial_\rho\vec{X}|$.
We also introduce the functional space with respect to the
evolution curve $\Gamma(t)$ as
\begin{equation}
L^2(\mathbf{I})=\{u: \mathbf{I}\rightarrow \mathbb{R}, \;\text{and} \int_{\Gamma(t)}|u(s)|^2 ds
=\int_{\mathbf{I}} |u(s(\rho,~t))|^2 \partial_\rho s\, d\rho <+\infty \},
\end{equation}
equipped with the $L^2$ inner product
\be
\big(u,v\big)_{\Gamma(t)}:=\int_{\Gamma(t)}u(s)v(s)\,ds=
\int_{\mathbf{I}}u(s(\rho,t))v(s(\rho,t)) \partial_\rho s\,d\rho,\qquad \forall\;u,v\in L^2(\mathbf{I}),
\ee
for any scalar (or vector) functions. Moreover, define the Sobolev spaces
\begin{subequations}
\begin{align}
H^1(\mathbf{I})&=\{u: \mathbf{I}\rightarrow \mathbb{R}, \;u\in L^2(\mathbf{I}),\text{and} \; \partial_\rho u\in L^2(\mathbf{I})\},\\
H^1_0(\mathbf{I})&=\{u: \mathbf{I}\rightarrow \mathbb{R}, \;u\in H^1(\mathbf{I}),\text{and}\;u(0)=u(1) =0\},
\end{align}
\end{subequations}
and denote $\mathbb{X}:=H^1(\mathbf{I})\times H_0^1(\mathbf{I})$.

Multiplying a test function $\psi\in H^1(\mathbf{I})$ to \eqref{eqn:2disonew1}, integrating over $\Gamma(t)$, integration by parts
and noting the zero-mass flux condition \eqref{eqn:weakBC3}, we obtain
\begin{eqnarray}\label{varf1}
0&=&\Bigl(\vec{n}\cdot \partial_{t}\vec{X}-\partial_{ss}\kappa,\psi\Bigr)_{\Gamma(t)}\nonumber\\
&=&\Bigl(\vec{n}\cdot \partial_{t}\vec{X},\psi\Bigr)_{\Gamma(t)}+\Bigl(\partial_{s}\kappa,
\partial_s\psi\Bigr)_{\Gamma(t)}-\left.\left(\psi\partial_{s}\kappa\right)
\right|_{\rho=0}^{\rho=1}\nonumber\\
&=&\Bigl(\vec{n}\cdot \partial_{t}\vec{X},\psi\Bigr)_{\Gamma(t)}+\Bigl(\partial_{s}\kappa,
\partial_s\psi\Bigr)_{\Gamma(t)}.
\end{eqnarray}
Similarly, using the dot-product with a test function $\boldsymbol{\omega}=(\omega_1,~\omega_2)^T\in \mathbb{X}$
 to \eqref{eqn:2disonew2}, integrating over $\Gamma(t)$, integration by parts
 and noting the relaxed contact angle condition \eqref{eqn:Robin}, we obtain
\begin{eqnarray}\label{varf2}
0&=&\Bigl(\kappa\,\vec{n}+\partial_{ss}\vec{X},\boldsymbol{\omega}
\Bigr)_{\Gamma(t)}\nonumber\\
&=&\Bigl(\kappa,~\vec n\cdot\boldsymbol{\omega}\Bigr)_{\Gamma(t)} -\Bigl(\partial_s\vec X,~\partial_s\boldsymbol{\omega}\Bigr)_{\Gamma(t)} + (\partial_s\vec X\cdot\boldsymbol{\omega})\Big|_{\rho=0}^{\rho=1}\nonumber\\
&=&\Bigl(\kappa,~\vec n\cdot\boldsymbol{\omega}\Bigr)_{\Gamma(t)}-\Bigl(\partial_s\vec X,~\partial_s\boldsymbol{\omega}\Bigr)_{\Gamma(t)} +(\partial_sx\,\omega_1)\Big|_{\rho=0}^{\rho=1}\\
&=&\Bigl(\kappa,~\vec n\cdot\boldsymbol{\omega}\Bigr)_{\Gamma(t)}-\Bigl(\partial_s\vec X,~\partial_s\boldsymbol{\omega}\Bigr)_{\Gamma(t)} - \frac{1}{\eta}\Bigr[\frac{dx_c^l(t)}{dt}\,\omega_1(0) + \frac{dx_c^r(t)}{dt}\omega_1(1)\Bigr] + \sigma\left[\omega_1(1)-\omega_1(0)\right].  \nonumber
\end{eqnarray}

Combining \eqref{varf1} and \eqref{varf2}, we obtain
a new variational formulation for the sharp-interface model (\ref{eqn:2diso}) with boundary conditions (\ref{eqn:weakBC1})-(\ref{eqn:weakBC3}) and initial condition \eqref{eqn:initial} as follows: Given an initial open curve $\Gamma(0)=\vec X(\rho,0)\in\mathbb{X}$ with $\vec X(\rho,0)=\vec X_0(L_0\rho)=\vec X_0(s)$ and set $x_c^l(0)=x_0(s=0)<x_c^r(0)=x_0(s=L_0)$, for $t>0$, find its evolution curves $\Gamma(t):=\vec X(\cdot,t)=(x(\cdot,t),~y(\cdot,t))^T\in \mathbb{X}$, and the curvature $\kappa(\cdot,~t)\in H^1(\mathbf{I})$ such that
\begin{subequations}
\label{eqn:2dweak}
\begin{align}
\label{eqn:2dweak1}
&\Bigl(\vec n\cdot\partial_t\vec X,~\psi\Bigr)_{\Gamma(t)} + \Bigl(\partial_s\kappa,~\partial_s\psi\Bigr)_{\Gamma(t)}=0,\qquad\forall \psi\in H^1(\mathbf{I}),\\[0.5em]
\label{eqn:2dweak2}
&\Bigl(\kappa,\vec n\cdot\boldsymbol{\omega}\Bigr)_{\Gamma(t)}-\Bigl(\partial_s\vec X,~\partial_s\boldsymbol{\omega}\Bigr)_{\Gamma(t)} - \frac{1}{\eta}\Bigr[\frac{dx_c^l(t)}{dt}\,\omega_1(0) + \frac{dx_c^r(t)}{dt}\omega_1(1)\Bigr]\nonumber\\
&\qquad\qquad\qquad\qquad + \;\sigma\,\Bigl[\omega_1(1) - \omega_1(0)\Bigr] = 0,\quad\forall\boldsymbol{\omega}=(\omega_1,~\omega_2)^T\in\mathbb{X};
\end{align}
\end{subequations}
where we adopt $x_c^l(t)=x(\rho=0,t)$ and $x_c^r(t)=x(\rho=1,t)$ and
assume that they satisfy $x_c^l(t)\le x_c^r(t)$.

\subsection{Area/mass conservation and energy dissipation}

For the variational problem \eqref{eqn:2dweak}, we have
\begin{prop}[mass conservation and energy dissipation]
Let $\Bigl(\vec X(\cdot,~t),~\kappa(\cdot,~t)\Bigr)$ be a solution of the variational problem \eqref{eqn:2dweak}. Then, the total area/mass of the film is conserved during the evolution, i.e.,
\begin{equation}\label{massvp}
A(t)\equiv A(0) = \int_0^{L_0}y_0(s)\partial_sx_0(s)\;ds,\qquad t\geq 0,
\end{equation}
and the total free energy of the system is decreasing during the evolution, i.e,
\begin{equation}\label{energydvp}
W(t)\leq W(t^\prime)\leq W(0) = L_0 - \sigma(x_c^r(0) - x_c^l(0)),\qquad t\geq t^\prime\geq 0.
\end{equation}
\end{prop}

\begin{proof}
Differentiating the left equation in \eqref{eqn:totalmass} with respect to  $t$, integrating by parts and noting \eqref{eqn:weakBC1}, we have
\begin{eqnarray}\label{changAt}
\frac{d}{dt}A(t)&=&\frac{d}{dt}\int_0^{L(t)}y( s,t)\partial_s x(s,t)\;ds\nonumber\\
& = &\frac{d}{dt}\int_0^1 y(\rho,~t)\partial_\rho x(\rho,~t)\;d\rho = \int_0^1 (\partial_t y \partial_\rho x + y \partial_t\partial_\rho x)\;d\rho \nonumber\\
&=& \int_0^1 (\partial_t y \partial_\rho x - \partial_\rho y\partial_t x)d\rho +(y\partial_t x)\Big|_{\rho = 0}^{\rho = 1} = \int_{\Gamma(t)}\partial_t\vec X\cdot\vec n\;ds,\qquad t\ge0.
\end{eqnarray}
Taking the test function $\psi = 1$ in \eqref{eqn:2dweak1} and then plugging it  into \eqref{changAt}, we obtain
\begin{equation}
\frac{d}{dt}A(t):=\int_{\Gamma(t)}\partial_t\vec X\cdot\vec n\;ds  = -\int_{\Gamma(t)}\partial_s\kappa\partial_s\psi ds = 0,\qquad t\ge0,
\end{equation}
which immediately implies the mass conservation \eqref{massvp}.

Similarly, differentiating the right equation in \eqref{eqn:totalmass}
 with respect to  $t$, we get
\begin{eqnarray}\label{changwt}
\frac{d}{dt}W(t) &=& \frac{d}{dt} L(t) - \sigma\Bigl[\frac{dx_c^r(t)}{dt} - \frac{dx_c^l(t)}{dt}\Bigr] \nonumber\\
&=& \frac{d}{dt}\int_0^1|\partial_\rho\vec X|\;d\rho - \sigma \Bigl[\frac{dx_c^r(t)}{dt} - \frac{dx_c^l(t)}{dt}\Bigr]\nonumber\\
&=& \int_0^1 \frac{\partial_\rho\vec X\cdot(\partial_\rho\partial_t\vec X)}{|\partial_\rho\vec X|}\;d\rho - \sigma \Bigl[\frac{dx_c^r(t)}{dt} - \frac{dx_c^l(t)}{dt}\Bigr]\nonumber\\
& = &\int_{\Gamma(t)}\partial_s\vec X\cdot(\partial_s\partial_t\vec X)\;ds  - \sigma \Bigl[\frac{dx_c^r(t)}{dt} - \frac{dx_c^l(t)}{dt}\Bigr],\qquad t\ge0.
\end{eqnarray}
Choosing the test functions $\psi = \kappa$ and $\boldsymbol{\omega} = \partial_t\vec X$ in \eqref{eqn:2dweak1} and \eqref{eqn:2dweak2}, respectively, and then inserting them into \eqref{changwt}, and noticing that $\omega_1(0)=\frac{dx_c^l(t)}{dt}$ and $\omega_1(1)=\frac{dx_c^r(t)}{dt}$, we obtain
\begin{equation}
\frac{d}{dt}W(t):=-\Bigl(\partial_s\kappa,~\partial_s\kappa\Bigr)_{\Gamma(t)} -  \frac{1}{\eta}\Bigl[\Bigl(\frac{dx_c^l(t)}{dt}\Bigr)^2+
\Bigl(\frac{dx_c^r(t)}{dt}\Bigr)^2\Bigr]\leq 0,\qquad t\ge0,
\end{equation}
which immediately implies the energy dissipation \eqref{energydvp}.
\end{proof}

\section{A semi-discretization in space}

In this section, we present a semi-discretization
of the variational formulation \eqref{eqn:2dweak} in space by using the continuous piecewise linear elements and show area/mass conservation
and energy dissipation as well as long-time behavior of the semi-discretization.

\subsection{The discretization}

Let $N>0$ be a positive integer, and denote the mesh size $h=1/N$,
grid points as $\rho_{j}=jh$ for $j=0,1,\ldots,N$, and subintervals
$I_j=[\rho_{j-1},\rho_{j}]$ for $j=1,2,\ldots,N$. Then
a uniform partition of the interval $\mathbf{I}$ is given as
$\mathbf{I}=[0,1]=\bigcup_{j=1}^{N}I_j$.
 Define the finite element subspaces as
\begin{equation*}
\begin{split}
&\mathbb{K}^h:=\{u\in C(\mathbf{I}):\;u\mid_{I_{j}}\in \mathcal{P}_1,\quad \forall \, j=1,2,\ldots,N\}\subseteq H^1(\mathbf{I}),\\
&\mathbb{X}^h := \mathbb{K}^h \times\mathbb{K}_0^h \qquad
\hbox{with} \qquad \mathbb{K}_0^h = \mathbb{K}^h\cap H_0^1(\mathbf{I}),
\end{split}
\end{equation*}
where $\mathcal{P}_1$ denotes the space of the polynomials with degree at most $1$.

Let $\Gamma^h(t)=\vec{X}^h(\cdot,t)\in \mathbb{X}^h$ be the numerical approximation of the solution
$\Gamma(t)=\vec{X}(\cdot,t)$ of the variational problem
\eqref{eqn:2dweak}. Then $\{\Gamma^h(t)\}_{t\ge0}$ are polygonal curves consisting of ordered line segments, and we  always assume that they satisfy
\begin{equation}\label{hjt987}
\min_{1\le j\le N} |\vec h_j(t)|>0, \qquad \hbox{with}
\quad \vec h_j(t)=\vec X^h(\rho_j,t)-\vec X^h(\rho_{j-1},t), \qquad j=1,2,\ldots, N, \qquad  t\geq 0,
\end{equation}
where $|\vec h_j(t)|$ denotes the length of the vector $\vec h_j(t)$ for $j=1,2,\ldots,N$. We note that the unit tangential vector $\boldsymbol{\tau}^h$ and normal vector $\vec n^h$ of the curve $\Gamma^h(t)$ are constant vectors on each interval $I_j$ with possible discontinuities or jumps at nodes $\rho_j$, and they can be easily computed as
\begin{equation}
\boldsymbol{\tau}^h|_{I_j}=\frac{\vec h_j}{|\vec h_j|}:=\boldsymbol{\tau}^h_j,\qquad
\vec{n}^h|_{I_j}
=-({\boldsymbol{\tau}}^h_j)^{\perp}=-\frac{(\vec h_j)^\perp}{|\vec h_j|}:=\vec{n}^h_j,\qquad 1\leq j\leq N,
\label{eqn:semitannorm}
\end{equation}
where $(\cdot)^\perp$ denotes the clockwise rotation by $90$ degrees.

For two piecewise continuous scalar or vector functions $u$ and $v$ defined on the interval $I$ with possible jumps at the nodes $\{\rho_j\}_{j=1}^{N-1}$,  we can define the mass lumped inner product $\big(\cdot,\cdot\big)_{\Gamma^h(t)}^h$ over $\Gamma^h(t)$ (which is an approximation of $\big(\cdot,\cdot\big)_{\Gamma^h(t)}$ by adopting the composite trapezoidal quadrature) as
\begin{equation}
\big(u,~v\big)_{\Gamma^h(t)}^h:=\frac{1}{2}\sum_{j=1}^{N}\Big|\vec h_j\Big|\Big[\big(u\cdot v\big)(\rho_j^-)+\big(u\cdot v\big)(\rho_{j-1}^+)\Big],
\label{eqn:semiproduct}
\end{equation}
where $u(\rho_j^\pm)=\lim\limits_{\rho\to \rho_j^\pm} u(\rho)$.

Take $\Gamma^h(0)=\vec{X}^h(\cdot,~0)\in \mathbb{X}^h $
such that $\vec{X}^h(\rho=\rho_j,~0)=\vec{X}_0(s=s_j^0)$ with $s_j^0=jL_0/N=L_0\rho_j$ for $j=0,1,\ldots,N$. Then a semi-discretization of the variational formulation \eqref{eqn:2dweak} by continuous piecewise linear elements can be stated as follows: Given the initial curve $\Gamma^h(0)=\vec X^h(\cdot,0)$ and set $x_c^l(0)=x_0(s=0)<x_c^r(0)=x_0(s=L_0)$, for $t>0$, find the evolution curves $\Gamma^h(t)=\vec X^h(\cdot,~t)=(x^h(\cdot,t),y^h(\cdot,t))^T\in \mathbb{X}^h$, and the curvature $\kappa^h(\cdot,~t)\in \mathbb{K}^h$ such that
\begin{subequations}
\label{eqn:2dsemi}
\begin{align}
\label{eqn:2dsemi1}
&\Bigl(\vec n^h\cdot\partial_t\vec X^h,~\psi^h\Bigr)_{\Gamma^h(t)}^h + \Bigl(\partial_s\kappa^h,~\partial_s\psi^h\Bigr)_{\Gamma^h(t)}=0,\qquad\forall \psi^h\in \mathbb{K}^h,\\[0.5em]
\label{eqn:2dsemi2}
&\Bigl(\kappa^h,\vec n^h\cdot\boldsymbol{\omega}^h\Bigr)_{\Gamma^h(t)}^h-\Bigl(\partial_s\vec X^h,~\partial_s\boldsymbol{\omega}^h\Bigr)_{\Gamma^h(t)} - \frac{1}{\eta}\Bigr[\frac{dx_c^l(t)}{dt}\,\omega_1^h(0) + \frac{dx_c^r(t)}{dt}\omega_1^h(1)\Bigr] \nonumber \\
&\qquad\qquad\qquad+ \sigma\,\Bigl[\omega_1^h(1) - \omega_1^h(0)\Bigr] = 0,\quad\forall\boldsymbol{\omega}^h=(\omega_1^h,~\omega_2^h)^T
\in\mathbb{X}^h;
\end{align}
\end{subequations}
where again we adopt $x_c^l(t)=x^h(\rho_0=0,t)$  and $x_c^r(t)=x^h(\rho_N=1,t)$ and assume that they satisfy $x_c^l(t)\le x_c^r(t)$.

We remark here that when $\eta=+\infty$, Eq.~\eqref{eqn:2dsemi2} will reduce to the scheme for the boundary condition with the isotropic Young's angle \eqref{eqn:young}. Similar schemes in the isotropic case for a single open curve evolution can be found in the literature~\citep[see][]{Barrett07b, Barrett10}.

\subsection{Area/mass conservation and energy dissipation}

For simplicity, denote $\vec X^h(\rho_j,~t)= (x_j(t),~y_j(t))^T$ for $j=0,1,\ldots,N$, then the total area/mass $A^h(t)$ and free energy $W^h(t)$ of the discrete polygonal curve $\vec X^h(t)$ can be written as
\begin{equation}
A^h(t):=\frac{1}{2}\sum_{j=1}^N\big(x_j-x_{j-1}\big)\big(y_j+y_{j-1}\big),\quad W^h(t):= \sum_{j=1}^N|\vec h_j(t)| - \sigma \big[x_c^r(t) - x_c^l(t)\big],\quad  t\ge0.
\label{eqn:semimassenergy}
\end{equation}
Similar to the previous works (e.g., see Theorem 88 in~\citet{Barrett20}), we can prove the area/mass conservation and energy dissipation properties for the semi-discretization~\eqref{eqn:2dsemi}.
\begin{prop}
\label{prop:semimassenergy}
Let $\Bigl(\vec X^h(\cdot,~t),~\kappa^h(\cdot,~t)\Bigr)$ be the solution of the semi-discretization \eqref{eqn:2dsemi}, then the total area/mass of the film is conserved, i.e.,
\begin{equation}\label{massconsd}
A^h(t)\equiv A^h(0)=\frac{1}{2}\sum_{j=1}^N\big[x_0(s_j)-x_0(s_{j-1})
\big]\big[y_0(s_j)+y_0(s_{j-1})\big]>0,\quad t\ge0,
\end{equation}
and the total free energy of the system is decreasing during the evolution, i.e.,
\begin{equation}\label{energdsd}
W^h(t)\le W^h(t^\prime)\le W^h(0)=\sum_{j=1}^N|\vec h_j(0)|-\sigma(x_c^r(0)-x_c^l(0)),\quad t\ge t^\prime\ge 0.
\end{equation}
\end{prop}

\begin{proof}
Noting $y_0=y_N = 0$, we can reformulate the left equation in \eqref{eqn:semimassenergy} as
\begin{eqnarray}\label{Ahtcon}
A^h(t)&=&\frac{1}{2}\sum_{j=1}^N\big(x_j-x_{j-1}\big)\big(y_j+y_{j-1}\big)
\nonumber\\
& =& \frac{1}{2}\sum_{j=1}^N\big(x_jy_j - x_{j-1}y_{j-1}\big) + \frac{1}{2}\sum_{j=1}^N\big(x_jy_{j-1} - x_{j-1}y_j\big)\nonumber\\
&=&\frac{1}{2}\sum_{j=1}^{N-1}x_jy_j - \frac{1}{2}\sum_{j=2}^Nx_{j-1}y_{j-1} + \frac{1}{2}\sum_{j=1}^N\big(x_jy_{j-1} - x_{j-1}y_j\big) \nonumber\\
&=&  \frac{1}{2}\sum_{j=1}^N\big(x_jy_{j-1} - x_{j-1}y_j\big).
\end{eqnarray}
Differentiating \eqref{Ahtcon} with respect to $t$, we have
\begin{eqnarray}\label{Ahtcond}
\frac{d}{dt}A^h(t)&=&\frac{1}{2}\sum_{j=1}^N\Bigl(\dot{x}_jy_{j-1} - \dot{x}_{j-1}y_j\Bigr) + \sum_{j=1}^N\Bigl(x_j\dot{y}_{j-1} - x_{j-1}\dot{y}_j\Bigr)\nonumber\\
&=&\frac{1}{2}\sum_{j=1}^N\Bigl(\dot{x}_jy_{j-1} - \dot{x}_jy_j + \dot{x}_{j-1}y_{j-1} - \dot{x}_{j-1}y_j\Bigr) \nonumber\\ &&\;+\; \frac{1}{2}\sum_{j=1}^N\Bigl(x_j\dot{y}_j - x_{j-1}\dot{y}_{j-1} + x_j\dot{y}_{j-1} - x_{j-1}\dot{y}_j\Bigr)\nonumber\\
&=& -\frac{1}{2}\sum_{j=1}^N\big(\dot{x}_j + \dot{x}_{j-1}\big)\big(y_{j-1} - y_j\big) + \sum_{j=1}^N\big(\dot{y}_j + \dot{y}_{j-1}\big)\big(x_j - x_{j-1}\big),
\end{eqnarray}
where, for simplicity, here we denote the time derivative of $f$ as $\dot{f}$. Combining \eqref{eqn:semitannorm}, \eqref{eqn:semiproduct} and
\eqref{Ahtcond}, we obtain
\begin{equation}\label{Ahtcondd}
\frac{d}{dt}A^h(t) = \frac{1}{2}\sum_{j=1}^N|\vec h_j(t)|\Bigl[\frac{d}{dt}\vec X_j^h + \frac{d}{dt}\vec X_{j-1}^h\Bigr]\cdot \vec n^h_j(t) = \Bigl(\partial_t\vec X^h,~\vec n^h\Bigr)_{\Gamma^h(t)}^h.
\end{equation}
Choosing the test function $\psi^h=1$ in \eqref{eqn:2dsemi1} and then
plugging it into \eqref{Ahtcondd}, we have
\begin{equation}
\frac{d}{dt}A^h(t)=-\big(\partial_s\kappa^h,~\partial_s 1\big)_{\Gamma^h(t)}=0,
\qquad t\ge0,
\label{eqn:semimass}
\end{equation}
which immediately implies the area/mass conservation \eqref{massconsd}.

Similarly, differentiating the right equation in \eqref{eqn:semimassenergy}
with respect to $t$, we obtain
\begin{eqnarray}\label{Whtsdh}
\frac{d}{dt}W^h(t) &=& \sum_{j=1}^N \frac{1}{|\vec h_j|}\bigl(\frac{d\vec h_j}{dt}\cdot\vec h_j\bigr) -\sigma \Big[\frac{dx_c^r(t)}{dt} - \frac{dx_c^l(t)}{dt}\Big]\nonumber\\
&=&\sum_{j=1}^N\Bigl[\frac{\vec X_j^h - \vec X_{j-1}^h}{|\vec X_j^h - \vec X_{j-1}^h|}\cdot\frac{\partial_t\vec X_j^h - \partial_t\vec X_{j-1}^h}{|\vec X_j^h - \vec X_{j-1}^h|}|\vec X_j^h - \vec X_{j-1}^h|\Bigr]-\sigma \Big[\frac{dx_c^r(t)}{dt} - \frac{dx_c^l(t)}{dt}\Big] \nonumber\\
&=& \Bigl(\partial_s\vec X^h,~\partial_s(\partial_t\vec X^h)\Bigr)_{\Gamma^h(t)} -\sigma \Big[\frac{dx_c^r(t)}{dt} - \frac{dx_c^l(t)}{dt}\Big].
\end{eqnarray}
Choosing the test functions $\psi^h = \kappa^h$ and $\boldsymbol{\omega}^h = \partial_t\vec X^h$ in \eqref{eqn:2dsemi1} and \eqref{eqn:2dsemi2}, respectively,  and then plugging them into \eqref{Whtsdh}, we obtain
\begin{equation}
\label{eqn:dsemienergy}
\frac{d}{dt}W^h(t) = -\Bigl(\partial_s\kappa^h,~\partial_s\kappa^h\Bigr)_{\Gamma^h(t)} - \frac{1}{\eta}\left[\left(\frac{dx_c^r(t)}{dt}\right)^2 + \left(\frac{dx_c^l(t)}{dt}\right)^2\right]\leq 0, \qquad t\geq 0,
\end{equation}
which immediately implies the energy dissipation \eqref{energdsd}.
\end{proof}

\subsection{Equal mesh distribution}
Define the mesh ratio indicator (MRI) $\Psi(t)$  of the curve $\Gamma^h(t)$ as
\begin{equation}\label{MRIpsit}
\Psi(t)=\frac{\max_{1\le j\le N}\;|\vec h_{j}(t)|}{\min_{1\le j\le N}\;|\vec h_{j}(t)|}
=\frac{\max_{1\le j\le N}\;\left|\vec X^h(\rho_j,t) - \vec X^h(\rho_{j-1},t)\right|}{\min_{1\le j\le N}\;\left|\vec X^h(\rho_j,t) - \vec X^h(\rho_{j-1},t)\right|},
\qquad t\ge0.
\end{equation}
Similar to the case of the PFEM for time evolution of a closed curve controlled by surface diffusion in \citet{Barrett2011,Barrett07}, we can show that the property of equal mesh distribution of $\Gamma^h(t)$ provided that
any three adjacent vertices on $\Gamma^h(t)$ are not co-linear.

\begin{prop}
\label{prop:semiequalmesh}
Let $\Bigl(\vec X^h(\cdot,~t),~\kappa^h(\cdot,~t)\Bigr)$ be the solution of the semi-discretization \eqref{eqn:2dsemi} and under the assumption \eqref{hjt987}, then we have
\begin{equation}
\label{eqn:semiequalmesh}
\left(|\vec h_{j+1}| - |\vec h_{j}|\right)
\left(|\vec h_j|\;|\vec h_{j+1}| - \vec h_{j}\cdot\vec h_{j+1}\right)=0, \qquad j=1,2,\ldots, N-1, \qquad t>0.
\end{equation}
Specifically, if any three adjacent vertices on $\Gamma^h(t)$ are not co-linear, then $\Gamma^h(t)$ satisfies equal mesh distribution, i.e.
\begin{equation} \label{MRIht789}
|\vec h_{1}(t)| = |\vec h_{2}(t)|=\ldots =|\vec h_{N}(t)| \quad \Leftrightarrow \quad \Psi(t)=1, \quad t>0.
\end{equation}
\end{prop}
\begin{proof}
This equidistribution property is a direct consequence of Eq.~\eqref{eqn:2dsemi2}. We omit the proof here, since the proof can be found from Remark 2.4 in \citet{Barrett07} or Theorem 62 in \citet{Barrett20}.
\end{proof}

\subsection{Long-time behavior and equilibrium state}

Define  the curvature variation indictor (CVI)
$D(t)$ of the curve $\Gamma^h(t)$ as
\begin{equation}
D(t)=\Bigl(\partial_s\kappa^{h},~\partial_s\kappa^h\Bigr)_{{\Gamma^h(t)}},
\qquad t\ge0.
\end{equation}
Then, we have


\begin{prop}
\label{prop:semi2}
Let $\Bigl(\vec X^h(\cdot,~t),~\kappa^h(\cdot,~t)\Bigr)$ be the solution of the semi-discretization \eqref{eqn:2dsemi} and assume
that $\min_{1\le j\le N}\;|\vec h_{j}(t)|>0$ for $t\in[0,\infty)$ and when $t\to+\infty$,
$\vec X^h(\cdot,~t)$ and $\kappa^h(\cdot,~t)$ converge to the equilibrium $\Gamma^e=\vec X^e(\rho)=(x^e(\rho),y^e(\rho))^T\in \mathbb{X}^h$ and $\kappa^e(\rho)\in\mathbb{K}^h$, respectively, satisfying $\min_{1\le j\le N}\;|\vec h_{j}^e|>0$ with $\vec h_{j}^e:=\vec X^e(\rho_j) - \vec X^e(\rho_{j-1})$ for $1\le j\le N$. Then we have
\begin{eqnarray}\label{lb1345}
&&\lim_{t\to +\infty} D(t)=D^e:=\Bigl(\partial_s\kappa^{e},~\partial_s\kappa^e\Bigr)_{{\Gamma^e}}=0,
\quad \Rightarrow \quad \kappa^e(\rho)\equiv\kappa^c,\quad 0\le \rho\le1,\\
&&\lim_{t\to +\infty} \Psi(t)=\Psi^e:=\frac{\max_{1\le j\le N}\;|\vec h_{j}^e|}{\min_{1\le j\le N}\;|\vec h_{j}^e|}
=\frac{\max_{1\le j\le N}\;\left|\vec X^e(\rho_j) - \vec X^e(\rho_{j-1})\right|}{\min_{1\le j\le N}\;\left|\vec X^e(\rho_j) - \vec X^e(\rho_{j-1})\right|}=1,
\label{lb13456}
\end{eqnarray}
where $\kappa^c\ne 0$ is a constant. Furthermore, denote $\theta_e^l$ and $\theta_e^r$ as the left and right contact angles, respectively,  of $\Gamma^e$, then there exists $h_0>0$ sufficiently small
 such that
\begin{equation}\label{thetarl156}
|\cos(\theta_e^l) - \sigma|\le Ch,\qquad |\cos(\theta_e^r) - \sigma| \le Ch,
\qquad 0<h\le h_0,
\end{equation}
where $C>0$ is a constant.
\end{prop}

\begin{proof}
Under the assumptions in this Proposition, letting $t\to+\infty$ in
the semi-discretization \eqref{eqn:2dsemi}, we obtain that
the equilibrium solution $\bigl(\vec X^e(\rho),~\kappa^e(\rho)\bigr)\in\mathbb{X}^h\times \mathbb{K}^h$ satisfies the following variational problem
\begin{subequations}
\label{eqn:equi}
\begin{align}
\label{eqn:equi1}
&\left(\partial_s\kappa^e,~\partial_s\psi^h\right)_{\Gamma^e}=0,\qquad\forall\psi^h\in\mathbb{K}^h,\\
&\left(\kappa^e,\vec n^e\cdot\boldsymbol{\omega}^h\right)_{\Gamma^e}^h - \left(\partial_s\vec X^e,~\partial_s\boldsymbol{\omega^h}\right)_{\Gamma^e} + \sigma \,\left[\omega_1^h(1) - \omega_1^h(0)\right]=0,\quad\forall\boldsymbol{\omega}^h
=(\omega_1^h,~\omega_2^h)^T\in\mathbb{X}^h,
\label{eqn:equi2}
\end{align}
\end{subequations}
where
\begin{equation}\label{hetauene}
\vec h_{j}^e=\vec X^e(\rho_j) - \vec X^e(\rho_{j-1}),\quad
\boldsymbol{\tau}^e|_{I_j}=\frac{\vec h_j^e}{|\vec h_j^e|}:=\boldsymbol{\tau}^e_j,\quad
\vec{n}^e|_{I_j}
=-({\boldsymbol{\tau}}^e_j)^{\perp}:=\vec{n}^e_j, \quad j=1,2\ldots,N.
\end{equation}
Choosing $\psi^h=\kappa^e(\rho)\in \mathbb{K}^h$ in \eqref{eqn:equi1} and noting $D(t)$ being a continuous function for $t>0$, we obtain
\begin{equation}\label{Dekbb}
\lim_{t\to\infty}D(t)=D^e=\bigl(\partial_s\kappa^e,~\partial_s\kappa^e\bigr)_{\Gamma^e}=0.
\end{equation}
Combining \eqref{Dekbb} and \eqref{eqn:semiproduct}, and noticing
$\kappa^e\in \mathbb{K}^h\subset C([0,1])$, we get
\begin{equation}
\label{eqn:semip1}
\kappa^e(\rho)\equiv\kappa^c,\qquad 0\leq \rho\leq 1.
\end{equation}
From the mass conservation \eqref{massconsd} and noticing \eqref{eqn:semimassenergy}, we get
\begin{equation}\label{Ae987}
A^e:=\frac{1}{2}\sum_{j=1}^N\left(x^e(\rho_j)-
x^e(\rho_{j-1})\right)\left(y^e(\rho_j)+y^e(\rho_{j-1})\right)=A^h(0)>0
\quad \Rightarrow \quad
y^e(\rho)\not\equiv 0.
\end{equation}
Choosing $\boldsymbol{\omega}^h=(0,~y^e)^T\in \mathbb{X}^h$ in \eqref{eqn:equi2}, noticing $\boldsymbol{\omega}^h(0)=\boldsymbol{\omega}^h(1)={\bf 0}$ and noting \eqref{eqn:semip1}, we have
\begin{equation}\label{kappc986}
0=\left(\kappa^e,\vec n^e\cdot\boldsymbol{\omega}^h\right)_{\Gamma^e}^h-
\left(\partial_s\vec X^e,~\partial_s\boldsymbol{\omega^h}\right)_{\Gamma^e}
=\kappa^c\left(1,\vec n^e\cdot\boldsymbol{\omega}^h\right)_{\Gamma^e}^h
- \left(\partial_sy^e,~\partial_sy^e\right)_{\Gamma^e}.
\end{equation}
Combining \eqref{kappc986} and \eqref{Ae987}, noting $y^e\in \mathbb{K}_0^h$,
we obtain
\be\label{kappc876}
\kappa^c\left(1,\vec n^e\cdot\boldsymbol{\omega}^h\right)_{\Gamma^e}^h=
 \left(\partial_sy^e,~\partial_sy^e\right)_{\Gamma^e}>0\quad \Rightarrow
\quad \kappa^c\ne0.
\ee

From \eqref{eqn:equi}, similar to the proof in the Proposition 3.2, we can prove
\begin{equation}
\label{eqn:semiequalmeeqm}
\left(|\vec h_{j+1}^e| - |\vec h_{j}^e|\right)
\left(|\vec h_j^e|\;|\vec h_{j+1}^e| - \vec h_{j}^e\cdot\vec h_{j+1}^e\right)=0, \qquad j=1,2,\ldots, N-1.
\end{equation}
Now, we adopt the method of contradiction to show
\be\label{mc876}
\boldsymbol{\tau}^e_j=\frac{\vec h_j^e}{|\vec h_j^e|}\ne \boldsymbol{\tau}^e_{j+1}=\frac{\vec h_{j+1}^e}{|\vec h_{j+1}^e|}, \qquad
1\le j\le N-1.
\ee
Assume that \eqref{mc876} is not true, i.e. there exists an integer $j_0$ ($1\le j_0\le N-1$) such that
\be
\boldsymbol{\tau}_{j_0}^e=\frac{\vec h_{j_0}^e}{|\vec h_{j_0}^e|} = \boldsymbol{\tau}_{j_0+1}^e=\frac{\vec h_{j_0+1}^e}{|\vec h_{j_0+1}^e|}
\quad \Rightarrow \quad \vec n_{j_0}^e = \vec n_{j_0+1}^e.
\ee
Let $\boldsymbol{\omega}^h(\rho)\in \mathbb{X}^h$ such that
\begin{equation}\label{wh1234567}
\boldsymbol{\omega}^h(\rho=\rho_k)=\left\{\begin{array}{ll}
\vec n_{j_0}^e, &k=j_0,\\
{\bf 0}, &k\ne j_0,\\
\end{array}\right. \qquad k=0,1,\ldots, N,
\end{equation}
noting \eqref{hjt987}, then we have
\begin{equation}\label{psomh765}
\partial_s\boldsymbol{\omega^h}(\rho)=\vec n_{j_0}^e
\left\{\ba{ll}
1/|\vec h_{j_0}|, &\rho_{j_0-1}\le \rho<\rho_{j_0},\\
-1/|\vec h_{j_0+1}|, &\rho_{j_0}\le \rho< \rho_{j_0+1},\\
0,  &\hbox{otherwise}.
\ea\right.
\end{equation}
Combining \eqref{psomh765} and \eqref{hetauene}, noticing $\left.\partial_s\vec X^e\right|_{I_k}=\boldsymbol{\tau}^e_k$ for
$1\le k\le N-1$, we have
\begin{eqnarray}\label{khnhwh112}
\left(\partial_s\vec X^e,~\partial_s\boldsymbol{\omega}^h
\right)_{\Gamma^e}&=&\frac{1}{2}\left[2|\vec h_{j_0}^e|
\boldsymbol{\tau}^e_{j_0}\cdot\vec n_{j_0}^e/|\vec h_{j_0}^e|-
2|\vec h_{j_0+1}^e|
\boldsymbol{\tau}^e_{j_0+1}\cdot\vec n_{j_0+1}^e/|\vec h_{j_0+1}^e| \right]\nonumber\\
&=&\boldsymbol{\tau}^e_{j_0}\cdot\vec n_{j_0}^e-\boldsymbol{\tau}^e_{j_0+1}\cdot\vec n_{j_0+1}^e=0.
\end{eqnarray}
Inserting $\boldsymbol{\omega}^h(\rho)$ in \eqref{wh1234567} into
\eqref{eqn:equi2}, noticing \eqref{khnhwh112}, \eqref{eqn:semip1}, \eqref{eqn:semiproduct} and \eqref{hetauene}, and noting  $\boldsymbol{\omega}^h(0)=\boldsymbol{\omega}^h(1)={\bf 0}$,  we obtain
\begin{eqnarray}\label{kpaeb76}
0&=&\left(\kappa^e,\vec n^e\cdot\boldsymbol{\omega}^h\right)_{\Gamma^e}^h-\left(\partial_s\vec X^e,~\partial_s\boldsymbol{\omega}\right)_{\Gamma^e}
=\kappa^c\left(1,\vec n^e\cdot\boldsymbol{\omega}^h\right)_{\Gamma^e}^h\nonumber\\
&=&\frac{1}{2}\kappa^c
\left[|\vec h_{j_0}^e|\; \vec n^e_{j_0}\cdot \vec n^e_{j_0}
+|\vec h_{j_0+1}^e|\; \vec n^e_{j_0+1}\cdot \vec n^e_{j_0+1}\right]\nonumber\\
&=&\frac{1}{2}\kappa^c
\left[|\vec h_{j_0}^e|
+|\vec h_{j_0+1}^e| \right].
\end{eqnarray}
Thus we get $\kappa^c=0$, which contradicts with $\kappa^c\ne 0$ in
\eqref{kappc876}, and therefore \eqref{mc876} is valid. From \eqref{mc876}, we have
\be\label{tauj76543}
0\ne 1-\boldsymbol{\tau}^e_j\cdot \boldsymbol{\tau}^e_{j+1}=
1-\frac{\vec h_j^e}{|\vec h_j^e|}\cdot \frac{\vec h_{j+1}^e}{|\vec h_{j+1}^e|}=\frac{1}{|\vec h_j^e|\;|\vec h_{j+1}^e|}\left(|\vec h_j^e|\;|\vec h_{j+1}^e|-\vec h_j^e\cdot\vec h_{j+1}^e\right), \quad 1\le j\le N-1.
\ee
Combining \eqref{eqn:semiequalmeeqm} and \eqref{tauj76543}, we get
\be\label{eqn:semip27}
|\vec h_1^e|=|\vec h_2^e|=\ldots=|\vec h_N^e|,
\ee
which immediately implies \eqref{lb13456} by noting that $\Psi(t)$ is a continuous function for $t>0$.

Denoting $L^e:= |\Gamma^e|>0$ as the length of $\Gamma^e$ and noticing \eqref{eqn:semip27}, we have
\begin{equation}
\label{eqn:semip3}
|\vec h_j^e| = \frac{L^e}{N} = L^e h,\qquad  1\leq j\leq N.
\end{equation}
Choosing $\boldsymbol{\omega}^h(\rho)\in \mathbb{X}^h$ which satisfies
\begin{equation}\label{wh12345678}
\boldsymbol{\omega}^h(\rho=\rho_k)=\left\{\begin{array}{ll}
{\bf e}_1:=(1,0)^T, &k=0,\\
{\bf 0}, &1\le k\le N,\\
\end{array}\right.
\end{equation}
in \eqref{eqn:equi2}, carrying a similar computation as in \eqref{psomh765}, noting~\eqref{eqn:semip1} and \eqref{eqn:semiproduct}, \eqref{eqn:semip3},
we obtain
\bea
0&=&\left(\kappa^e,\vec n^e\cdot\boldsymbol{\omega}^h\right)_{\Gamma^e}^h-\left(\partial_s\vec X^e,~\partial_s\boldsymbol{\omega}^h\right)_{\Gamma^e}-\sigma\nonumber\\
&=&\frac{1}{2}\kappa^c\,|\vec h_1^e| \; \vec n^e_{1}\cdot {\bf e}_1+
\frac{1}{2} |\vec h_1^e| \bigl[\boldsymbol{\tau}^e_1 \cdot {\bf e}_1/|\vec h_1^e|+\boldsymbol{\tau}^e_1 \cdot {\bf e}_1/|\vec h_1^e|\bigr]-\sigma\nonumber\\
&=&\frac{1}{2}\kappa^cL^e (\vec n^e_{1}\cdot {\bf e}_1)h+\cos(\theta_e^l) -\sigma,
\eea
where we use the relation $\cos(\theta_e^l)=\boldsymbol{\tau}^e_1 \cdot {\bf e}_1$.
Thus there exists $h_0>0$ sufficiently small such that, when $0<h\le h_0$, we have
\be\label{thetarl1}
\left|\cos(\theta_e^l) - \sigma\right| = \left|-
\frac{1}{2}\kappa^cL^e (\vec n^e_{1}\cdot {\bf e}_1)h\right|
=\frac{1}{2} L^e\, |\kappa^c|\; |\vec n^e_{1}\cdot {\bf e}_1|\,h
\le Ch,
\ee
where $C>0$ is a constant. Similarly, we can prove the right inequality
in  \eqref{thetarl156}.
\end{proof}

\begin{remark}
The above Proposition~\ref{prop:semi2} shows that the equilibrium state
obtained by the semi-discretization \eqref{eqn:2dsemi} has the following properties: (i) it has constant curvature; (ii) it has equal mesh
distribution and the MRI $\Psi(t)\approx 1$ when $t\gg1$;
and  (iii) its contact angles converge  to
the Young's contact angle $\theta_i$ (i.e., theoretical equilibrium contact angle) linearly (or at first-order) with respect to the mesh size $h$.
\end{remark}

\section{A full-discretization}

In this section, based on the idea in~\citet{Barrett07b,Barrett10}, we present an energy-stable parametric finite element method
(ES-PFEM) to discretize the semi-discretization \eqref{eqn:2dsemi} by adopting the (semi-implicit) backward Euler method in time and show the well-posedness and
energy dissipation of the full discretization.

\subsection{The discretization}

Take $\tau>0$ as the uniform time step size and denote discrete time levels as
$t_m=m\tau$ for $m=0,1,\ldots$. Then, for $m\ge0$, let $\Gamma^m:=\vec X^m=\vec X^m(\rho)=(x^m(\rho),y^m(\rho))^T\in
\mathbb{X}^h$ be the numerical approximation
of the solution $\Gamma^h(t_m)=\vec X^h(\cdot,t_m)\in \mathbb{X}^h $  of the semi-discretization \eqref{eqn:2dsemi} at $t=t_m$.
Again, $\{\Gamma^m\}_{m\ge0}$ are polygonal curves consisting of ordered line segments. For each curve $\Gamma^m$, we note that the unit tangential vector $\boldsymbol{\tau}^m$ and normal vector $\vec n^m$ are constant vectors on each interval $I_j$ with possible discontinuities or jumps at nodes $\rho_j$, and they can be easily computed as
\begin{equation}
\boldsymbol{\tau}^m|_{I_j}=\frac{\vec h_j^m}{|\vec h_j^m|}:=\boldsymbol{\tau}^m_j=\bigl(\tau_{j,1}^{m},~\tau_{j,2}^{m}\bigr)^T,\quad
\vec{n}^m|_{I_j}
=-({\boldsymbol{\tau}}^m_j)^{\perp}:=\vec{n}^m_j=
\bigl(n_{j,1}^{m},~n_{j,2}^{m}\bigr)^T,\quad 1\leq j\leq N,
\label{eqn:semitannormfd}
\end{equation}
where $\vec h_j^m=\vec X^m(\rho_{j})-\vec X^m(\rho_{j-1})$ for $j=1,2,\ldots,N$.

By adopting the (semi-implicit) backward Euler method to discretize the semi-discretizaiton
\eqref{eqn:2dsemi} in time, we obtain an ES-PFEM as a full-discretization
of the variational problem \eqref{eqn:2dweak} (i.e., for the sharp interface model (\ref{eqn:2diso}) with boundary conditions (\ref{eqn:weakBC1})-(\ref{eqn:weakBC3}) and initial condition \eqref{eqn:initial}) as follows:  Given the initial curve $\Gamma^0:=\vec X^{0}=(x^{0}(\rho),~y^{0}(\rho))^T\in \mathbb{X}^h$ and set $x_l^0=x^0(\rho=0)<x_r^0=x^0(\rho=1)$, for $m\ge0$, find the evolution curves $\Gamma^{m+1}:=\vec X^{m+1}=(x^{m+1}(\rho),~y^{m+1}(\rho))^T\in \mathbb{X}^h$, and the curvature $\kappa^{m+1}=\kappa^{m+1}(\rho)\in \mathbb{K}^h$ such that
\begin{subequations}
\label{eqn:dis2d}
\begin{align}
\label{eqn:dis2d1}
&\Bigl(\frac{\vec X^{m+1}-\vec X^m}{\tau}\cdot\vec n^m,~\psi^h\Big)_{\Gamma^m}^h + \Bigl(\partial_s\kappa^{m+1},~\partial_s\psi^h\Bigr)_{\Gamma^m} = 0,\qquad\forall\psi^h\in\mathbb{K}^h,\\
&\Bigl(\kappa^{m+1},\vec n^m\cdot\boldsymbol{\omega}^h\Bigr)_{\Gamma^m}^h - \Bigl(\partial_s\vec X^{m+1},~\partial_s\boldsymbol{\omega}^h\Bigr)_{\Gamma^m} + \sigma\,\Bigl[\omega_1^h(1) - \omega_1^h(0)\Bigr]  \nonumber\\
&\qquad\qquad\;-\;\frac{1}{\eta\,\tau}\Bigl[\omega_1^h(0)(x^{m+1}_l-x_l^m) +\omega_1^h(1)(x^{m+1}_r- x^m_r) \Bigr]=0,\quad\forall\boldsymbol{\omega}^h=(\omega_1^h,~\omega_2^h)^T\in \mathbb{X}^h;
\label{eqn:dis2d2}
\end{align}
\end{subequations}
where we adopt $x_l^{m+1}=x^{m+1}(\rho=0)$  and $x_r^{m+1}=x^{m+1}(\rho=1)$
and assume they satisfy $x_l^{m+1}\le x_r^{m+1}$. We remark here that, at each time step, only a linear system
is to be solved in the above ES-PFEM \eqref{eqn:dis2d}, which can be efficiently solved by either the GMRES method or the sparse LU decomposition in the practical computation.

We remark that the proposed ES-PFEM is an extension of the works in \cite{Barrett07b,Barrett10} to the case with relaxed contact angles. When $\eta=+\infty$, Eq.~\eqref{eqn:dis2d2} will reduce to the scheme with the isotropic Young's angle (see Eq.~(2.34) for a single open curve in \citet{Barrett07b} and Eq.~(3.11) in \citet{Barrett10}).

\subsection{Well-posedness}

For the well-posedness of the full-discretization \eqref{eqn:dis2d}, we have

\begin{thm}[Well-posedness] Assume that the following conditions are satisfied:
\begin{equation}
(\text{i}).~~(n^{m}_{1,1})^2+(n^{m}_{N,1})^2> 0;\qquad (\text{ii}).~~\min_{1\le j\le N} |\vec h^m_j|=\min_{1\le j\le N}|\vec X^m(\rho_{j})-\vec X^m(\rho_{j-1})|>0.
\label{eqn:assumption}
\end{equation}
Then the full-discretization \eqref{eqn:dis2d} is well-posedness, i.e.,
there exists a unique solution $\bigl(\vec X^{m+1},~\kappa^{m+1}\bigr)\in\bigl(\mathbb{X}^h,~\mathbb{K}^h\bigr)$.
\end{thm}

\begin{proof} In order to show that the linear system \eqref{eqn:dis2d} has a unique solution, we need only to prove that the following homogeneous linear system only has the zero solution: Find $\bigl(\vec X^{m+1},~\kappa^{m+1}\bigr)\in\bigl(\mathbb{X}^h,~\mathbb{K}^h\bigr)$ such that
\begin{subequations}
\label{eqn:homodis2d}
\begin{align}
\label{eqn:homodis2d1}
&\Bigl(\vec X^{m+1}\cdot\vec n^m,~\psi^h\Bigr)_{\Gamma^m}^h + \tau\Bigl(\partial_s\kappa^{m+1},~\partial_s\psi^h\Bigr)_{\Gamma^m} = 0, \qquad \forall\psi^h\in\mathbb{K}^h,\\
&\Bigl(\kappa^{m+1},~\vec n^m\cdot\boldsymbol{\omega}^h\Bigr)_{\Gamma^m}^h - \Bigl(\partial_s\Vec X^{m+1},~\partial_s\boldsymbol{\omega}^h\Bigr)_{\Gamma^m} -\frac{x_l^{m+1}\omega_1^h(0)+x_r^{m+1}\omega_1^h(1) }{\tau \eta} = 0,\quad\forall\boldsymbol{\omega}^h\in \mathbb{X}^h.
\label{eqn:homodis2d2}
\end{align}
\end{subequations}
Choosing the test functions $\psi^h = \kappa^{m+1}$ and $\boldsymbol{\omega}^h=\vec X^{m+1}$ in \eqref{eqn:homodis2d1} and  \eqref{eqn:homodis2d2}, respectively,  and then subtracting the second one from the first one, we obtain
\begin{equation}
\tau\Bigl(\partial_s\kappa^{m+1},\partial_s\kappa^{m+1}\Bigr)_{\Gamma^m} + \Bigl(\partial_s\vec X^{m+1},~\partial_s\vec X^{m+1}\Bigr)_{\Gamma^m} + \frac{1}{\tau \eta}\Bigl[(x_l^{m+1})^2+(x_r^{m+1})^2\Bigr] = 0.
\end{equation}
This, together with $\vec X^{m+1}\in \mathbb{X}^h$, $\kappa^{m+1}\in \mathbb{K}^h$, $x_l^{m+1}=x^{m+1}(\rho=0)$ and $x_r^{m+1}=x^{m+1}(\rho=1)$, implies
\begin{equation}\label{xlxrmp1}
x_l^{m+1} = x_r^{m+1}=0, \qquad \vec X^{m+1}(\rho)\equiv0, \qquad \kappa^{m+1}(\rho)\equiv \kappa^c, \qquad 0\le \rho\le 1,
\end{equation}
where $\kappa^c$ is a constant to be determined.
Substituting \eqref{xlxrmp1} into \eqref{eqn:homodis2d2}, we have
\begin{equation} \label{nmgmh11}
0=\Bigl(\kappa^{c},~\vec n^m\cdot\boldsymbol{\omega}^h\Bigr)_{\Gamma^m}^h=\kappa^c\Bigl(\vec n^m,~\boldsymbol{\omega}^h\Bigr)_{\Gamma^m}^h,\quad\forall
\boldsymbol{\omega}^h\in\mathbb{X}^h.
\end{equation}
Choosing the test vector function $\boldsymbol{\omega}^h(\rho)\in\mathbb{X}^h$ such that
\begin{equation}\label{omgh11}
\boldsymbol{\omega}^h(\rho=\rho_j)=\left\{\begin{array}{ll}
\Bigl(n_{1,1}^{m}, ~0\Bigr)^T, &j=0,\\
\vec 0, &1\le j\le N-1,\\
\Bigl(n_{N,1}^{m},~0\Bigr)^T, &j=N,\\
\end{array}\right.
\end{equation}
and noting \eqref{eqn:semiproduct}, \eqref{omgh11} and \eqref{eqn:assumption}, we get
\begin{eqnarray}\label{omg1hcc}
\Bigl(\vec n^m,~\boldsymbol{\omega}^h\Bigr)_{\Gamma^m}^h&=&\frac{1}{2}\sum_{j=1}^{N}\big
|\vec h^m_j\big|\Big[\boldsymbol{\omega}^h(\rho_{j-1})\cdot\vec n_{j}^m+\boldsymbol{\omega}^h(\rho_{j})\cdot\vec n_{j}^m\Big]=\frac{1}{2}\big
|\vec h^m_1\big|\,\boldsymbol{\omega}^h(\rho_{0})\cdot\vec n_{1}^m
+\frac{1}{2}\big
|\vec h^m_N\big|\,\boldsymbol{\omega}^h(\rho_{N})\cdot\vec n_{N}^m \nonumber\\
&=&\frac{1}{2}\big
|\vec h^m_1\big|\,(n_{1,1}^{m})^2+
\frac{1}{2}\big
|\vec h^m_N\big|\,(n_{N,1}^{m})^2\ge C\left((n_{1,1}^{m})^2+(n_{N,1}^{m})^2\right) >0,
\end{eqnarray}
where $C=\frac{1}{2}\min\left\{\big
|\vec h^m_1\big|, \big
|\vec h^m_N\big|\right\}$.

Combining \eqref{nmgmh11} and \eqref{omg1hcc}, we get $\kappa^c=0$
and thus $\kappa^{m+1}(\rho)\equiv 0$ for $0\le \rho\le 1$.
Therefore, we prove that the corresponding homogeneous linear system \eqref{eqn:homodis2d} only has the zero solution, and consequently, the original inhomogeneous linear system \eqref{eqn:dis2d} has a unique solution.
\end{proof}

\medskip

We remark here that the condition \eqref{eqn:assumption} is in fact quite a natural assumption, which implies that: (i) either the first or last line segment of the polygonal curve $\Gamma^m$ is not parallel to the $x$-axis; and (ii) any two neighbouring mesh points along the polygonal curve $\Gamma^m$ are distinct.

\subsection{Energy dissipation}

 Denote
\begin{equation}\label{Wmm11}
W^m:=W(\Gamma^m)=|\Gamma^m|-\sigma (x_r^m-x_l^m), \qquad m=0,1.\ldots
\end{equation}
Similar to the previous works in the literature (see Theorem 2.3 in \citet{Barrett07} and Theorem 87 in~\citet{Barrett20}), we can prove

\begin{thm}[Unconditional energy stability]
\label{Th:2Dstable}
Let $\bigl(\vec X^{m+1},~\kappa^{m+1}\bigr)\in\bigl(\mathbb{X}^h,~\mathbb{K}^h\bigr)$ be the solution of \eqref{eqn:dis2d}, then the energy is decreasing during
the evolution, i.e.
\begin{equation}
\label{eqn:stability1}
W^{m+1}
\leq W^m\le W^0=W(\Gamma^0)=|\Gamma^0|-\sigma (x_r^0-x_l^0), \qquad m\ge0.
\end{equation}
Moreover, we have
\begin{equation}
\label{eqn:stability2}
\sum_{l=1}^{m+1}\left[\Bigl(\partial_s\kappa^{l},
~\partial_s\kappa^{l}\Bigr)_{\Gamma^{l-1}}   + \frac{1}{\eta} \left(\frac{x_l^{l}-x_l^{l-1}}{\tau}\right)^2+ \frac{1}{\eta}
\left(\frac{x_r^{l} - x_r^{l-1}}{\tau}\right)^2\right]\leq \frac{W^0-W^{m+1}}{\tau}, \quad m\ge0.
\end{equation}
\end{thm}

\begin{proof}
Choosing the test functions $\psi^h = \kappa^{m+1}$ and and $\boldsymbol{\omega}^h = \vec X^{m+1}-\vec X^m$  in \eqref{eqn:dis2d1} and
\eqref{eqn:dis2d2}, respectively,  and then subtracting the second one from the first one, we obtain
\begin{eqnarray}
&&\tau\Bigl(\partial_s\kappa^{m+1},~\partial_s\kappa^{m+1}\Bigr)_{\Gamma^m} + \Bigl(\partial_s\vec X^{m+1},~\partial_s(\vec X^{m+1}-\vec X^m)\Bigr)_{\Gamma^m} - \sigma\Bigl[x_r^{m+1}-x_l^{m+1} - (x_r^m - x_l^m)\Bigr]\nonumber\\
&&\qquad\qquad+\frac{1}{\eta\,\tau}\Bigl[(x_l^{m+1}-x_l^m)^2+(x_r^{m+1} - x_r^m)^2 \Bigr] =0.
\label{eqn:midstep}
\end{eqnarray}
Note that the following inequality holds (e.g., see Lemma 2.2 in \citet{Bansch05} or  Lemma 57 in \citet{Barrett20})
\begin{equation}\label{Gm1111}
|\Gamma^{m+1}|-|\Gamma^m|\le \Bigl(\partial_s\vec X^{m+1},~\partial_s(\vec X^{m+1}-\vec X^m)\Bigr)_{\Gamma^m}.
\end{equation}
Plugging \eqref{Gm1111} into \eqref{eqn:midstep} and noticing \eqref{Wmm11}, we get
\begin{eqnarray}\label{Gm2345}
W^{m+1}&=&|\Gamma^{m+1}| - \sigma(x_r^{m+1} - x_l^{m+1})\nonumber\\
&\le&W^{m+1}+ \tau\Bigl(\partial_s\kappa^{m+1},~\partial_s\kappa^{m+1}\Bigr)_{\Gamma^m} + \frac{1}{\eta\tau}\Bigl[ (x_l^{m+1}-x_l^m)^2+(x_r^{m+1} - x_r^m)^2 \Bigr]\nonumber\\
&\le&|\Gamma^m| - \sigma(x_r^m - x_l^m)=W^m,
\end{eqnarray}
which immediately implies the energy dissipation \eqref{eqn:stability1}.
Re-formulating \eqref{Gm2345} as
\begin{equation}\label{Gm3456}
\Bigl(\partial_s\kappa^{l},~\partial_s\kappa^{l}\Bigr)_{\Gamma^{l-1}} + \frac{1}{\eta} \left(\frac{x_l^{l}-x_l^{l-1}}{\tau}\right)^2+\frac{1}{\eta}\left(\frac{x_r^{l} - x_r^{l-1}}{\tau}\right)^2
\le \frac{W^{l-1}-W^{l}}{\tau},  \qquad 1\le l\le m+1,
\end{equation}
and summing \eqref{Gm3456} for $l=1,2,\ldots,m+1$, we immediately obtain
the inequality \eqref{eqn:stability2}.
\end{proof}

\subsection{Long-time behavior and equilibrium state}

Again, denote the mesh ratio indicator (MRI) $\Psi^m$ and the curvature variation indicator (CVI)
$D^m$ of the curve $\Gamma^m$ as
\begin{equation}
\Psi^m=\frac{\max_{1\le j\le N}\;\left|\vec X^m(\rho_j) - \vec X^m(\rho_{j-1})\right|}{\min_{1\le j\le N}\;\left|\vec X^m(\rho_j) - \vec X^m(\rho_{j-1})\right|}, \qquad
D^m=\Bigl(\partial_s\kappa^{m},~\partial_s\kappa^m\Bigr)_{{\Gamma^{m-1}}},
\qquad m\ge1.
\end{equation}

Similar to the proof in Section 3.4 (with the proof omitted here for brevity), we have

\begin{prop}
\label{prop:full1}
Let $\Bigl(\vec X^m(\cdot),~\kappa^m(\cdot)\Bigr)$ be the solution of the full-discretization by using the ES-PFEM \eqref{eqn:dis2d}  and assume
that $\min_{1\le j\le N}\;|\vec h_{j}^m|>0$ for $m=1,2,\ldots$ and when $m\to+\infty$,
$\vec X^m(\cdot)$ and $\kappa^m(\cdot)$ converge to the equilibrium $\Gamma^e=\vec X^e(\rho)=(x^e(\rho),y^e(\rho))^T\in \mathbb{X}^h$ and $\kappa^e(\rho)\in\mathbb{K}^h$, respectively, satisfying $\min_{1\le j\le N}\;|\vec h_{j}^e|>0$ with $\vec h_{j}^e:=\vec X^e(\rho_j) - \vec X^e(\rho_{j-1})$ for $1\le j\le N$. Then we have
\begin{eqnarray}\label{lb1345fd}
&&\lim_{m\to +\infty} D^m=D^e:=\Bigl(\partial_s\kappa^{e},~\partial_s\kappa^e\Bigr)_{{\Gamma^e}}=0,
\quad \Rightarrow \quad \kappa^e(\rho)\equiv\kappa^c,\quad 0\le \rho\le1,\\
&&\lim_{m\to +\infty} \Psi^m=\Psi^e:=\frac{\max_{1\le j\le N}\;|\vec h_{j}^e|}{\min_{1\le j\le N}\;|\vec h_{j}^e|}
=\frac{\max_{1\le j\le N}\;\left|\vec X^e(\rho_j) - \vec X^e(\rho_{j-1})\right|}{\min_{1\le j\le N}\;\left|\vec X^e(\rho_j) - \vec X^e(\rho_{j-1})\right|}=1,
\label{lb13456fd}
\end{eqnarray}
where $\kappa^c\ne 0$ is a constant. Furthermore, denote $\theta_e^l$ and $\theta_e^r$ as the left and right contact angles, respectively,  of $\Gamma^e$, then there exists $h_1>0$ sufficiently small such that
\begin{equation}\label{thetarl156fd}
|\cos(\theta_e^l) - \sigma|\le C_1h,\qquad |\cos(\theta_e^r) - \sigma| \le C_1h, \qquad 0<h\le h_1,
\end{equation}
where $C_1>0$ is a constant.
\end{prop}

%
%

\begin{remark}
Similar to the semi-discretization \eqref{eqn:2dsemi}, the equilibrium solution of the full-discretization \eqref{eqn:dis2d} also has the following properties: (i) it has constant curvature; (ii) it has equal mesh
distribution and the MRI $\Psi^m\approx 1$ when $m\gg1$;
and  (iii) its contact angles converge  to
the Young's contact angle $\theta_i$ (i.e., theoretical equilibrium contact angle) linearly (or at first-order) with respect to the mesh size $h$.
\end{remark}

\section{Numerical results}

In this section, we first introduce a new manifold distance between two curves and adopt it to compare convergence rates of the proposed ES-PFEM \eqref{eqn:dis2d} and the PFEM proposed in \citet{Bao17} for solid-state dewetting problems. Then we report the temporal evolution of the area, contact angles, total free energy and some other indicators by using the proposed ES-PFEM under different computational parameters. Subsequently, we discuss some properties of the numerical equilibria, i.e., some convergence results about the total area loss and numerical contact angles. Finally, we show some applications of the ES-PFEM to simulating island film evolution.

We remark that the contact line mobility $\eta$ precisely controls the relaxation rate of the dynamic contact angle $\theta_d$ to the Young's contact angle $\theta_i = \arccos\sigma \in (0,\pi)$, where $\sigma$ is the material constant, and the larger $\eta$ will accelerate the relaxation process (for more details, see~\citet{Wang15}). In the following numerical simulations, we always choose $\eta=100$ for simplicity. We take the following four different types of initial shapes with $C^0$, $C^1$, $C^{\infty}$-smooth and non-convex curves, respectively:
\begin{itemize}
\item Shape 1 ($C^0$-smooth): a $6\times 1$ rectangle;
\item Shape 2 ($C^1$-smooth): a $4\times 1$ rectangle together with two quarter of circles on its left and right sides;
\item Shape 3 ($C^\infty$-smooth): half an ellipse with semi-major axis $a_x=4$ and semi-minor axis $a_y=1$;
\item Shape 4 (non-convex): a polar curve given as $r(\theta) = 2 + \cos(6\theta)$, with $\theta\in[0,~\pi]$.
\end{itemize}

\subsection{A manifold distance between two curves}

\begin{figure}[!htp]
\centering
\includegraphics[width=0.8\textwidth]{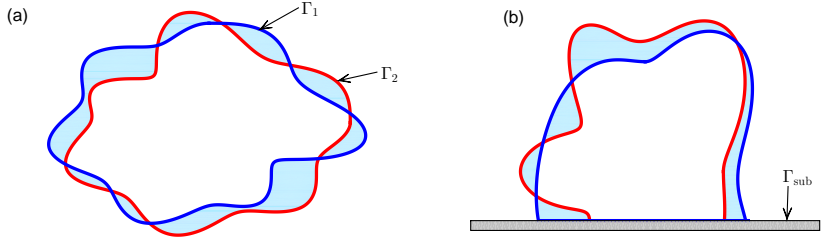}
\caption{An illustration of the manifold distance $M(\Gamma_1,\Gamma_2)$ between two curves $\Gamma_1$ and $\Gamma_2$ (colored in red and blue, respectively): (a) two closed curves, (b) two open curves on a flat substrate $\Gamma_{\rm sub}$, where ${\rm M}(\Gamma_1,\Gamma_2)$ is defined as the area of the symmetric difference region between two enclosed domains, i.e., the regions shaded in sky-blue.}
\label{fig:distance}
\end{figure}

As we know, it is a difficult and challenging problem on how to measure the difference between two curves in 2D or two surfaces in 3D. Some measures were used in the literature~~\citep[e.g.,][]{Bao17} by using the distance between the points on one curve (or surface) to the other curve (or surface). Inspired by the geometric quantity used for studying the quantitative stability for the Wulff shape of a given surface tension energy~\citep[see][]{Figalli10,Figalli17}, here we introduce a manifold distance between two curves. As shown in Fig.~\ref{fig:distance},
let $\Gamma_1$ and $\Gamma_2$ be two closed curves (cf. Fig.~\ref{fig:distance}(a)) or
two open curves on a flat substrate $\Gamma_{\rm sub}$ (cf. Fig.~\ref{fig:distance}(b)),
and we denote $\Omega_1$ and $\Omega_2$ as the inner regions enclosed by $\Gamma_1$ and $\Gamma_2$, respectively.
Then we introduce the manifold distance to measure the difference
between the two curves $\Gamma_1$ and $\Gamma_2$ by the area of the symmetric difference region between $\Omega_1$ and $\Omega_2$ (shown in the sky-blue region in Fig.~\ref{fig:distance})
\begin{equation}\label{MG1G2}
{\rm M}(\Gamma_1,~\Gamma_2):= |\left(\Omega_1\backslash \Omega_2 \right)\cup\left( \Omega_2\backslash\Omega_1\right)|=|\Omega_1|+|\Omega_2|-2|\Omega_1\cap\Omega_2|,
\end{equation}
where $|\Omega|$ denotes the area of the region $\Omega$.
We remark here that the definition in \eqref{MG1G2} can be very easily extended to the case of two open curves on a curved substrate in 2D and
the case of two closed surfaces or two open surfaces on a flat/curved substrate in 3D.

\medskip

For the manifold distance \eqref{MG1G2}, it is straightforward to show that:

\begin{prop}
\label{prop:mdd}
Let $\Gamma_1$, $\Gamma_2$ and $\Gamma_3$ be three closed curves or
open curves on a flat substrate. Then the manifold distance \eqref{MG1G2}
has the following properties:
\begin{itemize}
\item[(i)] (Symmetry). ${\rm M}(\Gamma_1,~\Gamma_2)={\rm M}(\Gamma_2,~\Gamma_1)$;
\item[(ii)] (Positivity).  ${\rm M}(\Gamma_1,~\Gamma_2)\ge0$, and ${\rm M}(\Gamma_1,~\Gamma_2)=0$ if and only if $\Gamma_1=\Gamma_2$;
\item[(iii)] (Triangle Inequality).  ${\rm M}(\Gamma_1,~\Gamma_2)\le {\rm M}(\Gamma_1,~\Gamma_3)+{\rm M}(\Gamma_2,~\Gamma_3)$.
\end{itemize}
\end{prop}

\begin{proof}
The properties (i)-(ii) are obvious by using the definition of the manifold distance.
In order to prove the property (iii), we first denote $\Omega_{A} = \left(\Omega_1\backslash \Omega_2 \right)\cup\left( \Omega_2\backslash\Omega_1\right)$, $\Omega_B=\left(\Omega_1\backslash \Omega_3 \right)\cup\left( \Omega_3\backslash\Omega_1\right)$ and $\Omega_C=\left(\Omega_2\backslash \Omega_3 \right)\cup\left( \Omega_3\backslash\Omega_2\right)$.
Then, we will show that $\Omega_A\subset \left(\Omega_B\cup\Omega_C\right)$.
For any given point $\vec x\in\Omega_A$, we have either $\vec x\in\Omega_1\backslash\Omega_2$ or $\vec x\in\Omega_2\backslash\Omega_1$.
Furthermore, by introducing another set $\Omega_3$, we only have the following four different cases:
\begin{itemize}
\item if $\vec x\in\Omega_1\backslash\Omega_2$ and $\vec x \in \Omega_3$, we have $\vec x \in\left(\Omega_3\backslash\Omega_2\right)$;
\item if $\vec x\in\Omega_1\backslash\Omega_2$ and $\vec x \not\in \Omega_3$, we have $\vec x \in\left(\Omega_1\backslash\Omega_3\right)$;
\item if $\vec x\in\Omega_2\backslash\Omega_1$ and $\vec x \in \Omega_3$, we have $\vec x \in\left(\Omega_3\backslash\Omega_1\right)$;
\item if $\vec x\in\Omega_2\backslash\Omega_1$ and $\vec x \not\in \Omega_3$, we have $\vec x \in\left(\Omega_2\backslash\Omega_3\right)$.
\end{itemize}
Therefore, by noticing the definitions of the sets $\Omega_B$ and $\Omega_C$, we immediately have proved that $\Omega_A\subset \left(\Omega_B\cup\Omega_C\right)$, and
$|\Omega_A|\leq|\Omega_B\cup\Omega_C|\leq |\Omega_B|+|\Omega_C|$.
\end{proof}

We remark that, to measure the convergence of numerical solutions in this paper, we only need to compute the manifold distance between polygonal curves in practice, because we use the piecewise linear curve to approximate the true curve solution. Since the polygonal curves consist of ordered line segments, we can easily determine the intersection points between two different polygonal curves, then calculate the enclosed area as shown in Fig.~\ref{fig:distance}.

\subsection{Convergent rates}

Let $\Gamma^m$ be the polygonal curve as an approximation of the
exact solution of the curve  $\Gamma(t=t_m)$, which is
obtained numerically at time $t=t_m$ under the mesh size $h$ by a
numerical method. In order to test the spatial accuracy of the proposed ES-PFEM~\eqref{eqn:dis2d}, we define the numerical error based on the manifold distance as
\begin{equation}
e_h(t_m):= {\rm M}(\Gamma^m,~\Gamma(t=t_m)),\qquad m=0,1,\cdots
\end{equation}

To compare the numerical convergence rates under the defined numerical errors $e_h(t_m)$ by using the ES-PFEM and the PFEM proposed in~\citet{Bao17}, we use the same computational set-ups for the two methods. The initial shape of the island film is chosen as the Shape '2' defined above, and $\sigma = \cos(5\pi/6)$. Here since we only focus on the spatial accuracy of the schemes, we
choose a very small time step $\tau = 1\times 10^{-6}$ for all cases in order to make the temporal error negligible compared to the spatial error. The `exact' solution of the curve $\Gamma(t)$ is obtained numerically under a very fine mesh size (we choose here $h=2^{-13}$) with the same time step $\tau$.

\begin{figure}[!htp]
\centering
\includegraphics[width=0.99\textwidth]{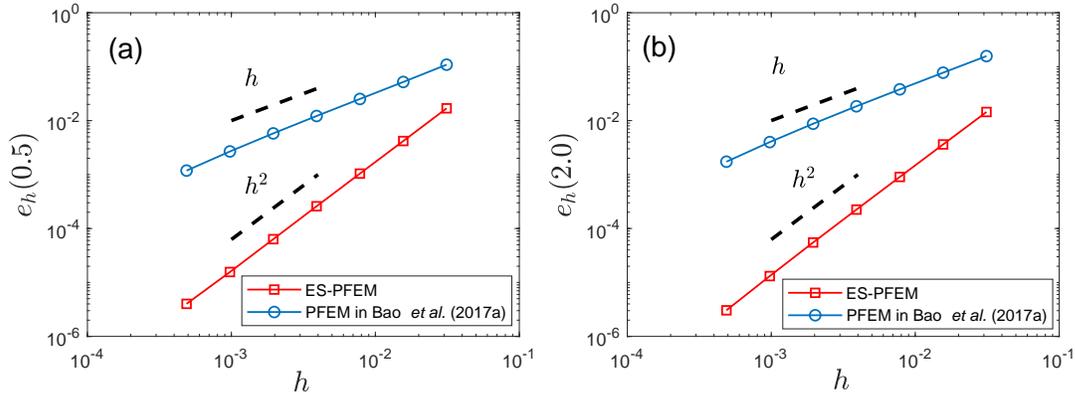}
\caption{Comparisons between the numerical convergence rates by using the proposed ES-PFEM
and the PFEM proposed by~\citet{Bao17} at two different times: (a) $t=0.5$, (b) $t=2.0$, where the material constant $\sigma=\cos(5\pi/6)$.}
\label{fig:order1}
\end{figure}

Fig.~\ref{fig:order1} depicts the log-log plots of the numerical errors $e_h(t_m)$ versus the mesh size $h$ by using the proposed ES-PFEM and the PFEM proposed by~\citet{Bao17} at time $t=0.5$ and $t=2.0$. We can clearly observe the numerical convergence rates under the manifold distance error for the two PFEMs. From these numerical results, we can draw the following conclusions: (i) the numerical errors for our proposed ES-PFEM under the same computational parameters can be significantly reduced compared to those obtained from the PFEM in~\citet{Bao17}; (ii) the convergence rate for the ES-PFEM can reach the second-order with respect to mesh size $h$, while for the PFEM in~\citet{Bao17}, the convergence rate is only first-order.
\begin{figure}[!htp]
\centering
\includegraphics[width=0.99\textwidth]{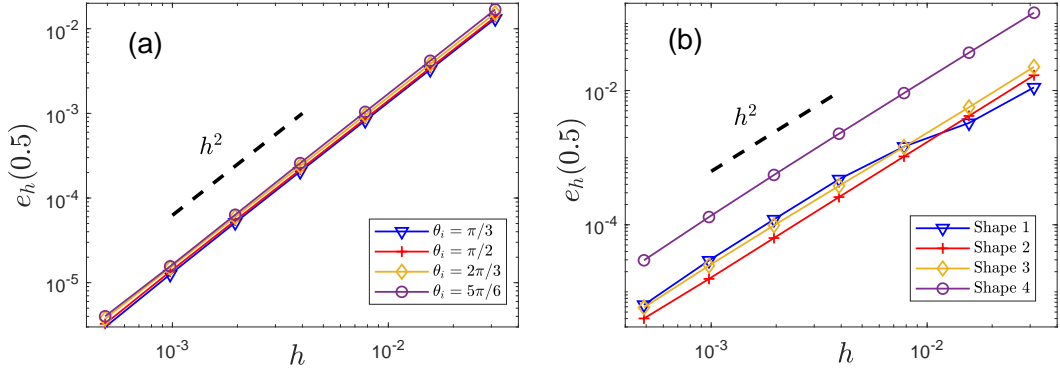}
\caption{The numerical errors as a function of mesh size $h$ by using the proposed ES-PFEM at time
$t=0.5$ (a) under four different Young's contact angles $\theta_i$, where the initial shape is chosen as the Shape `2'; (b) under four different initial shapes, where $\theta_i = 5\pi/6$.}
\label{fig:order2}
\end{figure}

Furthermore, to demonstrate that the second-order convergence rate in the sense of the manifold
distance error is independent of the initial set-ups, we perform ample numerical simulations for different Young's contact angles $\theta_i$ and different types of initial shapes.
The numerical errors of the ES-PFEM under different Young's contact angles $\theta_i$ and these different initial shapes at time $t=0.5$ are displayed in Fig.~\ref{fig:order2}(a),(b), respectively.
From the figure, we can clearly observe that the numerical errors as a function of mesh size $h$ behave almost the same for all the cases, regardless of different Young's contact angels $\theta_i$ and different initial shapes. The numerical convergence rates by using the proposed ES-PFEM are all second-order, which demonstrates the numerical convergence rate under the manifold distance error is very robust.

\subsection{Time evolution of area, contact angles, energy and indicators}

In the following, we present some numerical results about the temporal evolution of the total area of the film and the dynamical contact angle by using the proposed ES-PFEM. As we know, the exact solution of the sharp-interface model satisfies the area/mass conservation, and that the dynamical contact angle $\theta_d$ will converge to the Young's contact angle $\theta_i$. Here, we would like to investigate how these quantities numerically converge to their theoretical values as we refine the mesh size $h$. First, we define the relative area loss as
\begin{equation}
\Delta A^h(t):=\frac{A^h(t) - A^h(0)}{A^h(0)},\qquad t\geq 0,
\end{equation}
where $A^h(0)$ is the total area of the initial shape. As shown in Fig.~\ref{fig:fig6}(a), the area loss mainly happens at the beginning of the evolution, then almost keeps a value until approaching its equilibrium state. This value of the area loss can be significantly decreased when we refine the mesh size. Moreover, from Fig.~\ref{fig:fig6}(b), we also observe the numerical convergence between the dynamic contact angle and Young's angle $\theta_i$ in the long time when we refine the mesh size from $h=1/64$ to $h=1/512$.

\begin{figure}[!htp]
\centering
\includegraphics[width=0.99\textwidth]{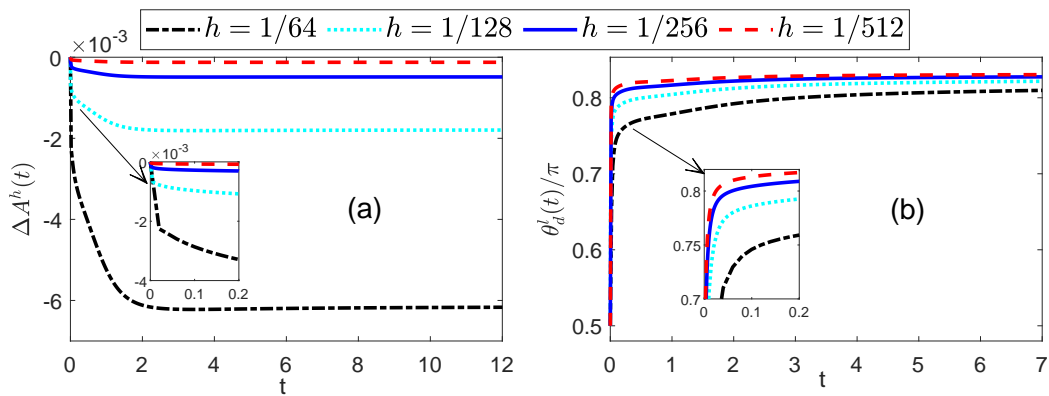}
\caption{(a) The temporal evolution of the relative area loss $\Delta A^h(t)$ under four different mesh sizes; (b) the temporal evolution of the left dynamic contact angle $\theta_d^l(t)$ converging to the Young's angle $\theta_i$ under four different mesh sizes. The initial shape is chosen as the Shape `2', and the Young's angle is chosen as $\theta_i=5\pi/6$, and $\tau = \frac{2048}{25}h^2$.}
\label{fig:fig6}
\end{figure}

\begin{figure}[!htp]
\centering
\includegraphics[width=0.99\textwidth]{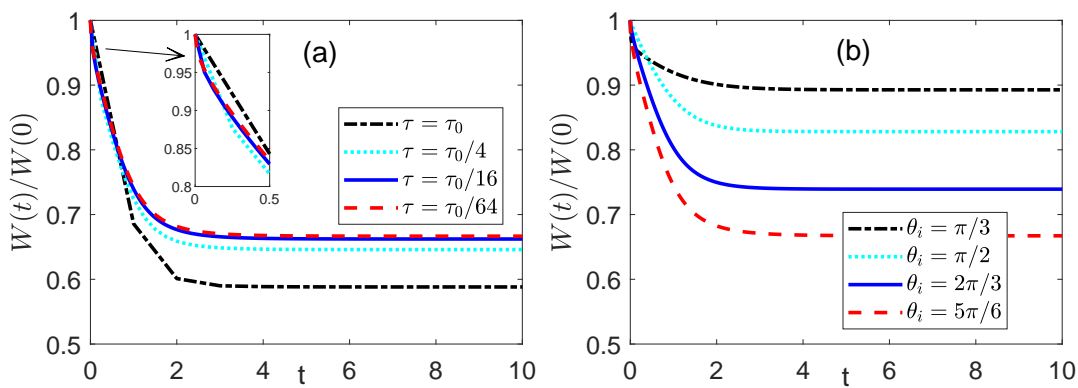}
\caption{The temporal evolution of the normalized total free energy (a) under different time steps $\tau$, where $\tau_0=1$, and $h = 1/128$ and $\theta_i=5\pi/6$; (b) under different Young's angles $\theta_i$, where $h=1/128$ and $\tau=0.01$. The initial shape is chosen as the Shape `2'.}
\label{fig:energy}
\end{figure}

Then, we numerically investigate the temporal evolution of the total free energy produced by
the ES-PFEM. As shown in Fig.~\ref{fig:energy}, we clearly observe that the normalized total free energy $W(t)/W(0)$ is always decreasing during the evolution, no matter how the time step and Young's angles are chosen. We note that although we choose a very large time step (e.g., $\tau_0=1$), the numerical solution still satisfies the energy dissipation very well. This
adequately reflects that the ES-PFEM is unconditionally energy-stable from the practical simulations.

\begin{figure}[!htp]
\centering
\includegraphics[width=0.99\textwidth]{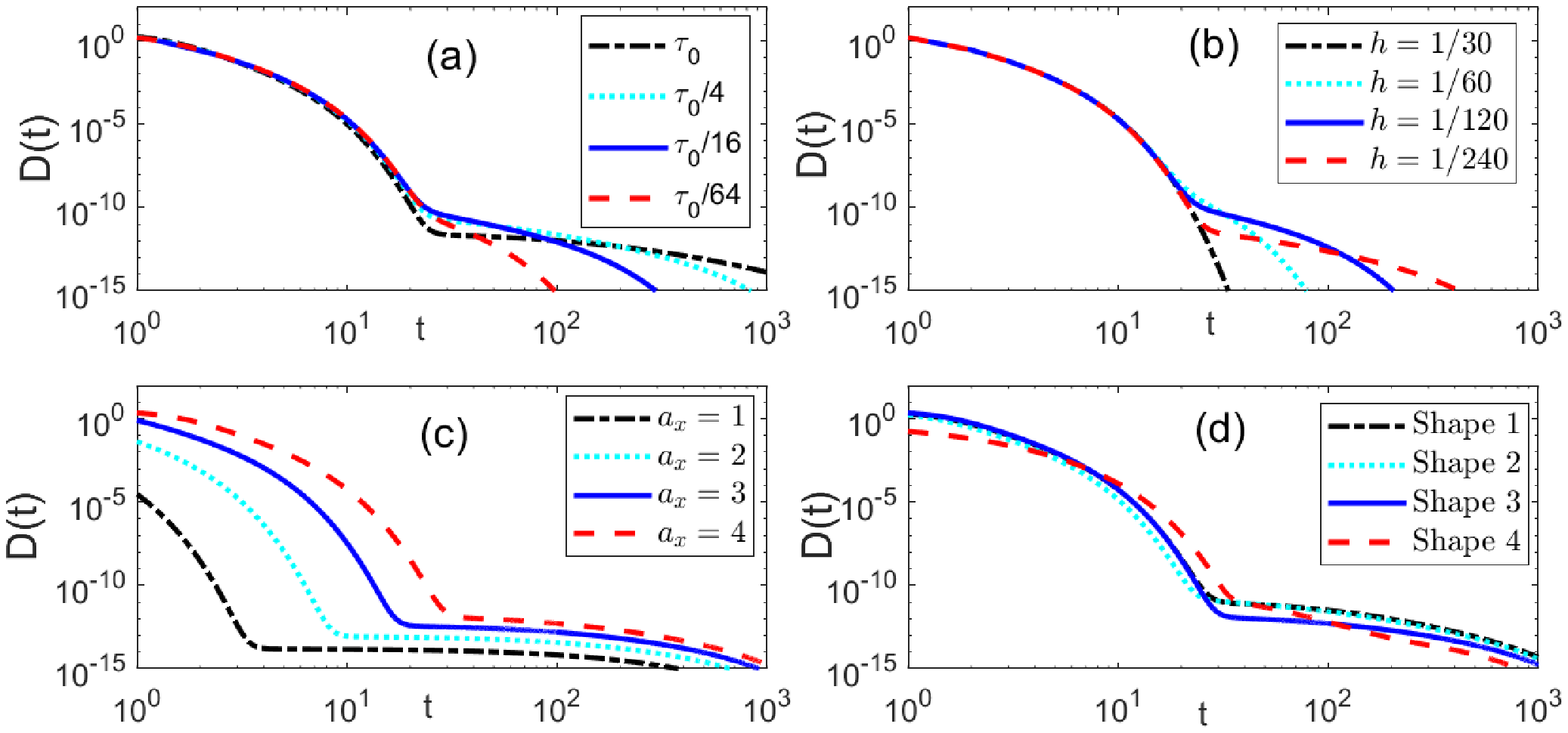}
\caption{The temporal evolution of the CVI function $D(t)$: (a) under different time steps, where $\tau_0=0.2$, $h=1/128$ and $\theta_i=5\pi/6$; (b) under different mesh sizes, where $\tau=0.01$ and $\theta_i=5\pi/6$; (c) under initial shapes of half an ellipse with different
semi-major axes $a_x$, where the semi-minor axis is fixed as $a_y = 1$, $\tau=0.1$, $h=1/128$ and $\theta_i=5\pi/6$; (d) under the Shapes `1-4', where $\tau=0.1$, $h=1/128$ and $\theta_i=5\pi/6$.
For (a)-(b), the initial shapes are both chosen as the Shape `2'.}
\label{fig:Indicator}
\end{figure}

\begin{figure}[!htp]
\centering
\includegraphics[width=0.99\textwidth]{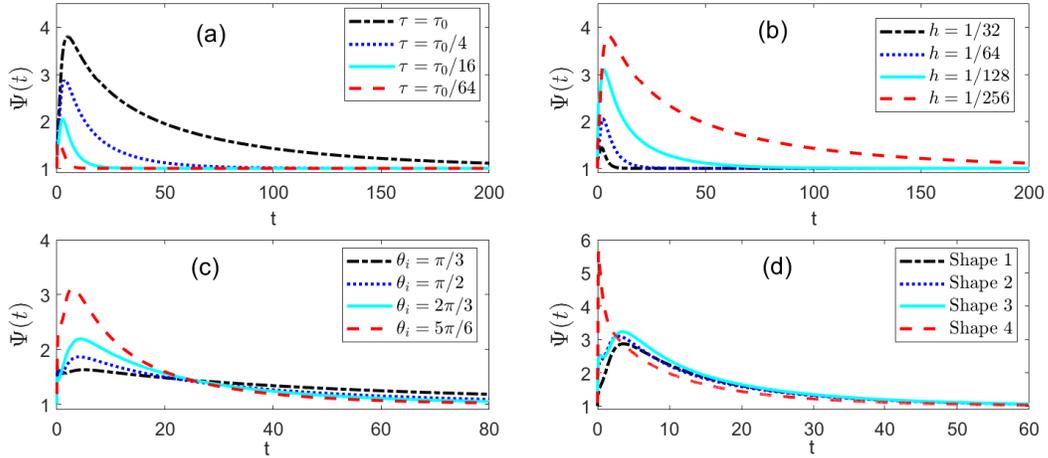}
\caption{The temporal evolution of the MRI function $\Psi(t)$:
(a) under different time steps with $\tau_0=2\times 10^{-2}$, $h=1/128$ and $\theta_i=5\pi/6$;
(b) under different mesh sizes with $\tau=5\times 10^{-3}$ and $\theta_i=5\pi/6$;
(c) under different Young's angles $\theta_i$, where $\tau=5\times 10^{-3}$, $h=1/128$;
(d) under the Shapes `1-4' with $\tau=5\times 10^{-3}$, $h=1/128$ and $\theta_i=5\pi/6$;
For (a)-(c), the initial shapes are all chosen as the Shape `1'.}
\label{fig:psii}
\end{figure}

Furthermore, we investigate the temporal evolution of the curvature variation indicator (CVI) function $D(t)$ by using ES-PFEM under different computational parameters and initial shapes. As shown in Fig.~\ref{fig:Indicator}, we observe that the indicator function $D(t)$ will eventually decrease as close to zero as possible up to the machine precision, no matter how the time steps, mesh sizes and initial shapes are chosen. When the indicator function $D(t)$ decreases to zero, it implies that the evolution curve attains the equilibrium shape.

Next, we perform ample numerical simulations to investigate the temporal evolution of the mesh ratio indicator (MRI) function $\Psi(t)$ by choosing different time steps, mesh sizes, Young's angles, and initial shapes. As clearly shown in Fig.~\ref{fig:psii}, no matter how the computational parameters are chosen, we can observe that at the beginning the function $\Psi(t)$ quickly increases to a critical value (which is no more than
$6$), then gradually decreases in a long time, finally converges to $1$, which implies that the proposed ES-PFEM could make mesh points equally distribute along the curve in the long time limit. Numerical results also indicate that the value of the upper bound for $\Psi(t)$ could be decreased by choosing smaller time steps and large mesh sizes (shown in Fig.~\ref{fig:psii}(a)-(b)).

\subsection{Some properties in computing equilibria}

We report some convergence results about the proposed ES-PFEM when applying it to computing the
equilibria of solid-state dewetting problems. First, we define that equilibrium state has reached, if $\frac{W^{m}-W^{m+1}}{\tau}\leq\epsilon$ satisfies, where $\epsilon$ is a small parameter and we choose $\epsilon=10^{-8}$ here. The numerical convergence rate of the total area and the contact angle for the equilibrium state as a function of mesh size $h$ are depicted in Fig~\ref{fig:Fig5c}. From the figure, we can clearly observe that: (i) the convergence rate of the area loss in the equilibrium state is second-order; (ii) the convergence rate of the error between the numerical equilibrium contact angle and theoretical equilibrium contact angle (i.e., Young's angle) is first-order.

\begin{figure}[!htp]
\centering
\includegraphics[width=0.9\textwidth]{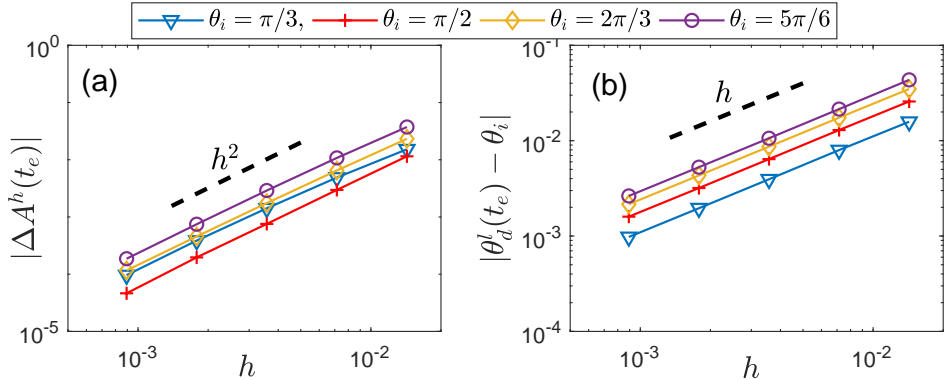}
\caption{(a) The area loss between the numerical equilibrium shapes and the theoretical equilibrium shape as a function of mesh size; (b) the numerical errors between the numerical equilibrium contact angle $\theta_d^l(t_e)$ and the Young's angle $\theta_i$ as a function of mesh size, where $t_e$ represents the arriving time when the equilibrium state first reaches, and the time step is chosen as $\tau = \frac{2048}{25}h^2$.}
\label{fig:Fig5c}
\end{figure}

We have numerically demonstrated that the convergence rate of the ES-PFEM under the manifold distance error is second-order. The manifold distance error is defined as the area of symmetric difference region, and it has some
natural relations with the area loss between the numerical solution and exact solution. So,
it is expected that the convergence rate of the area loss in the equilibrium state is also second-order, and the above numerical results have verified this expectation. Meanwhile, in the equilibrium state, as stated in Proposition~\ref{prop:semi2}, our numerical results also have demonstrated that the convergence rate of the error as a function of mesh size $h$ between the numerical equilibrium contact angle and Young's angle is first-order, i.e., $|\theta_e-\theta_i|=\mathcal{O}(h)$, where $\theta_e$ represents the numerical equilibrium contact angle.

\subsection{Applications to island evolution}
\begin{figure}[!htp]
\centering
\includegraphics[width=0.99\textwidth]{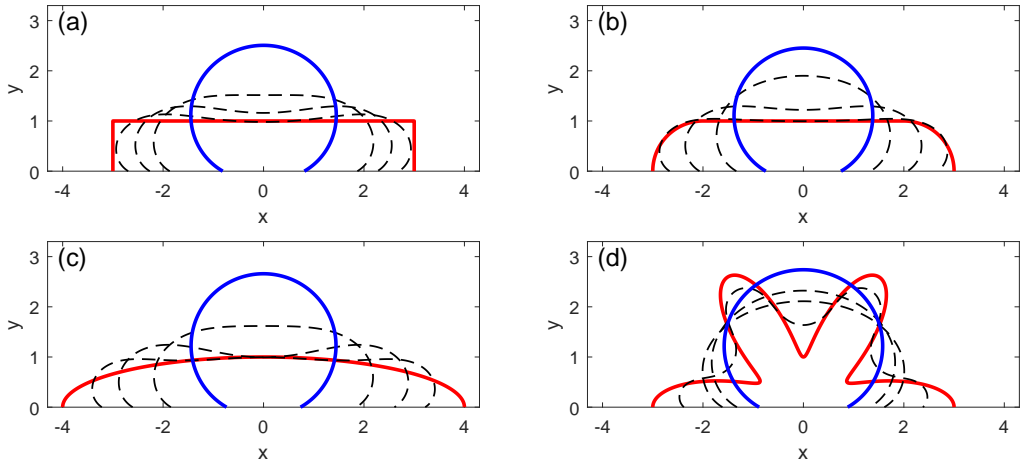}
\caption{ Several steps in the evolution of small islands (shown in red line) towards the equilibrium shape (shown in blue) under different initial shapes, where the initial shapes are respectively chosen as: (a) Shape `1'; (b) Shape '2'; (c) Shape '3'; (d) Shape '4'. The computational parameters are chosen as: $\theta_i=5\pi/6$, $h=1/560$, $\tau = 3.125\times 10^{-4}$.}
\label{fig:fig7}
\end{figure}

Finally, we present some numerical simulations about morphology evolution of an island film towards its equilibrium shape during solid-state dewetting. As shown in Fig.~\ref{fig:fig7},
we choose four different initial shapes (red solid line), and depict several evolution snapshots
(black dashed lines) until they arrive at the equilibrium shapes (blue solid line). From the figure, we can clearly observe that these curves eventually evolve into a perfect circular arc which intersects with the substrate at the same contact angle $\theta_i=5\pi/6$, no matter how the initial shapes are chosen. Then, by choosing different Young's angle $\theta_i$, as shown in Fig.~\ref{fig:figtheta}, we clearly observe that the equilibrium shape is still a circular arc, but its equilibrium contact angle changes according to different $\theta_i$.

\begin{figure}[!htp]
\centering
\includegraphics[width=0.99\textwidth]{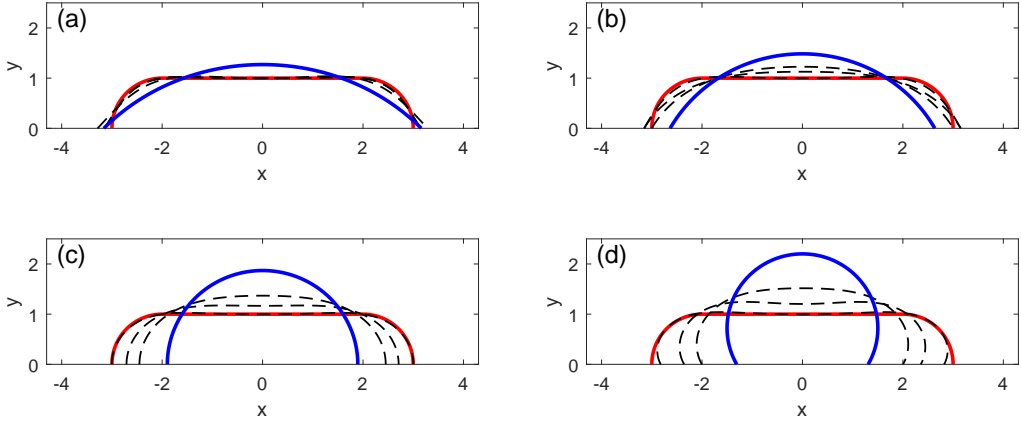}
\caption{Several steps in the evolution of small islands (shown in red line) towards the equilibrium shape (shown in blue) under four different Young's angles: (a) $\theta_i=\pi/4$; (b) $\theta_i=\pi/3$; (c)  $\theta_i=\pi/2$; (d) $\theta_i=2\pi/3$, where the initial shape is chosen as the Shape '2', and  the computational parameters are chosen as: $h=1/140$, $\tau=5\times10^{-3}$.}
\label{fig:figtheta}
\end{figure}

\section{Conclusions}

We proposed an energy-stable parametric finite element method (ES-PFEM) for solving the sharp-interface model about simulating solid-state dewetting of thin films with isotropic surface energies in 2D. The sharp-interface model describes the evolution of an open curve controlled by surface diffusion and contact line (point) migration. By reformulating the relaxed contact angle condition of the sharp-interface model into a Robin-type boundary condition, a new variational formulation was proposed such that the relaxed contact angle condition can be naturally imposed on the variational formulation. Then, we discretized the variational formulation by using piecewise linear elements.
Furthermore, we proved that the solution of the proposed PFEM is well-posed, and satisfies
the energy dissipation.
Compared to the PFEM presented in \citet{Bao17}, we showed that
the proposed ES-PFEM can attain the second-order convergence rate, rather than
the first-order under the manifold distance error for this type of open curve evolution problems
arising in solid-state dewetting. Besides, the newly proposed scheme has very good energy
stability, area/mass conservation, and long-time equal mesh distribution properties. Our future works will consider to extend the ES-PFEM to simulating solid-state dewetting of thin films with anisotropic surface energies as well as 3D problems~\citep[see][]{Jiang19a,Zhao19b}.

\section*{Acknowledgements}

This work was supported by the National Natural Science Foundation of China No.~11871384 (W.J.), Natural Science Foundation of Hubei Province No.~2018CFB466 (W.J.), and the Academic Research Fund of the Ministry of Education of Singapore grant No.~R-146-000-296-112 (Q.Z.\& W.B.). Part of the work was done when the authors were visiting the Institute of Mathematical Science at the National University of Singapore in 2020.

\bibliographystyle{IMANUM-BIB}
\bibliography{thebib}
\end{document}